\documentclass[11pt]{amsart}
\usepackage[top=4.2cm, bottom=4cm, left=2.4cm, right=2.4cm]{geometry}
\usepackage[utf8]{inputenc}
\usepackage[USenglish]{babel}
\usepackage[T1]{fontenc} 
\usepackage{mathrsfs}
\usepackage{comment}
\usepackage{wasysym}
\usepackage{tikz-cd}
\usepackage{mathtools}
\usepackage{amsmath}
\usepackage{amssymb,epsfig}
\usepackage{amsthm}
\usepackage[absolute]{textpos}
\usepackage{mathtools,emptypage}
\usepackage{mathtools}
\usepackage{faktor}

\mathtoolsset{showonlyrefs}  

\usepackage[bookmarks=true]{hyperref}
\usepackage{xcolor}
\hypersetup{
    colorlinks,
    linkcolor={red!50!black},
    citecolor={blue!50!black},
    urlcolor={blue!80!black},
}

\newcommand{\ppi}{{\mbox{\boldmath$\pi$}}}
\newcommand{\sfd}{{\sf d}}
\newcommand{\restr}[1]{\lower2pt\hbox{$|_{#1}$}}
\newcommand{\ca}[1]{\overset{\circ}{{#1}}}

\newcommand{\Pe}{{\rm P}}

\newcommand{\Kliminf}{K\kern-3pt-\kern-2pt\mathop{\rm lim\,inf}\limits}  
\newcommand{\Lip}{\mathop{\rm Lip}\nolimits}          
\renewcommand{\d}{{\mathrm d}}

\newcommand{\e}{{\rm{e}}}                          
 \newcommand{\X}{{\rm X}}
 \newcommand{\Y}{{\rm Y}}
  \newcommand{\R}{{\mathbb R}}

\newcommand{\limi}{\varliminf}
\newcommand{\lims}{\varlimsup}

\newcommand{\diam}{{\rm diam}}
\newcommand{\dist}{{\rm dist}}
\newcommand{\lip}{{\rm lip}}
\newcommand{\loc}{{\rm loc}}

\newcommand{\mm}{\mathfrak m}                                

\def\Xint#1{\mathchoice
{\XXint\displaystyle\textstyle{#1}}%
{\XXint\textstyle\scriptstyle{#1}}%
{\XXint\scriptstyle\scriptscriptstyle{#1}}%
{\XXint\scriptscriptstyle\scriptscriptstyle{#1}}%
\!\int}
\def\XXint#1#2#3{{\setbox0=\hbox{$#1{#2#3}{\int}$ }
\vcenter{\hbox{$#2#3$ }}\kern-.6\wd0}}

\def\dashint{\Xint-}

\newtheorem{theorem}{Theorem}[section]

\newtheorem{lemma}[theorem]{Lemma}
\newtheorem{proposition}[theorem]{Proposition}

\theoremstyle{definition}
\newtheorem{remark}[theorem]{Remark}
\newtheorem{definition}[theorem]{Definition}

\newtheorem{example}[theorem]{Example}
\newtheorem{question}[theorem]{Question}

\numberwithin{equation}{section}
\begin{document}

\title[Closed BV-extension and $W^{1,1}$-extension sets]{Closed BV-extension and $W^{1,1}$-extension sets}

\author[Emanuele Caputo]{Emanuele Caputo}

\address[Emanuele Caputo]{Mathematics Institute, Zeeman Building, University of Warwick, Coventry, CV4 7AL, United Kingdom}
\email{emanuele.caputo@warwick.ac.uk}

\author[Jesse Koivu]{Jesse Koivu}
\address[Jesse Koivu]{University of Jyvaskyla \\
         Department of Mathematics and Statistics \\
         P.O. Box 35 (MaD) \\
         FI-40014 University of Jyvaskyla \\
         Finland}
\email{jesse.j.j.koivu@jyu.fi}

\author[Danka Lu\v{c}i\'c]{Danka Lu\v{c}i\'c}
\address[Danka Lu\v{c}i\'c]{University of Jyvaskyla \\
         Department of Mathematics and Statistics \\
         P.O. Box 35 (MaD) \\
         FI-40014 University of Jyvaskyla \\
         Finland}
\email{danka.d.lucic@jyu.fi}

\author[Tapio Rajala]{Tapio Rajala}
\address[Tapio Rajala]{University of Jyvaskyla \\
         Department of Mathematics and Statistics \\
         P.O. Box 35 (MaD) \\
         FI-40014 University of Jyvaskyla \\
         Finland}
\email{tapio.m.rajala@jyu.fi}

\date{\today}



\begin{abstract}
This paper studies the relations between extendability of different classes of Sobolev $W^{1,1}$ and $BV$ functions from closed sets in general metric measure spaces. 
Under the assumption that the metric measure space satisfies a weak $(1,1)$-Poincar\'e inequality and measure doubling, we prove further properties for the extension sets. In the case of the Euclidean plane, we show that compact finitely connected $BV$-extension sets are always also $W^{1,1}$-extension sets. This is shown via a local quasiconvexity result for the complement of the extension set.
\end{abstract}

\maketitle

\section{Introduction}

The study of extension properties of sets (with respect to a certain family of functions) in a given ambient space has a long history and importance both in the mathematical theory and its applications. Starting from the fundamental works in the Euclidean ambient space (e.g.\ \cite{Calderon, Shvartsman, Jones}), the study continued over the years in several directions: considering different classes of functions, different assumptions on the sets and on the ambient space. In order to consider several cases of interest all together, the setting of metric measure spaces turned out to be an appropriate choice of the ambient space to work in. The latter is also due to the fact that the theories of different functional spaces, such as \(BV\) or Sobolev spaces, have been well established and thoroughly studied (see \cite{Mir03,Ambrosio-DiMarino14,HKST15,AILP24}).

When the first generalizations  from the Euclidean to a more general metric measure space setting had been done, e.g.\ in \cite{HKT2008b}, some more structural assumptions on the spaces were needed. For instance, the validity of a (local) Poincar\' e inequality and the doubling property of a measure; for some developments in this setting see \cite{Panu}. From the point of view of sets in consideration, many of the extension properties in the latter setting are understood and studied for domains, i.e.\  open connected sets. In the full generality, i.e.\ without additional structural assumptions on a metric measure space, the study of the extension properties for domains (and more in general, for open sets) has been pursued by three of the authors in \cite{CKR23}.
\smallskip

A particular question relevant in the study of extension properties that we will focus our discussion on is the following: \emph{when does the extension property with respect to a certain function space implies the extension property with respect to another function space?} In the present paper, we will address the above question in the case of \emph{closed sets} in general \emph{metric measure spaces}, where the extension properties we consider are relative to \(BV\)- or \(W^{1,1}\)-type function spaces.

To be more precise, let a metric measure space \((\X,\sfd, \mm)\) be given.
Given a Borel set $E \subset \X$, we say that $E$ is a $BV$-extension set if there exist $C\ge 1$ and a (not necessarily linear) map $T \colon BV(E) \to BV(\X)$ such that $T u=u$ on $E$ for every $u \in BV(E)$ and $\|Tu \|_{BV(\X)} \le C\|u \|_{BV(E)}$.
For the precise definition of the $BV$ space on a Borel set $E$ and the notion of an extension set, we refer the reader respectively to Definitions \ref{def:BV_on_borel} and \ref{def:extensions}. 
The same question mentioned above has been investigated in this generality in the case of open sets recently in \cite{CKR23}, generalizing earlier results \cite{KoskelaMirandaShanmugalingam2010,Burago,Baldi} in Euclidean and PI spaces.
Our study is motivated by \cite{KoivuLucicRajala2023}, where three of the authors proved the following general approximation result.

\begin{theorem}\label{thm:extensionsub}
 Let $(\X,\sfd,\mm)$ be a metric measure space. Let $\Omega \subset \X$ be a bounded open set. Then for every $\varepsilon >0$ there exists a closed set $G \subset \Omega$ such that $\mm(\Omega\setminus G) < \varepsilon$ and so that the zero extension gives a bounded operator from $BV(G)$ to $BV(\X)$.
\end{theorem}

The reason why $G$ in Theorem \ref{thm:extensionsub} is closed is due to the proof method in \cite{KoivuLucicRajala2023}. There the set $G$ was obtained via minimization of a functional involving perimeter,
penalized by a term measuring the distance from $\Omega$. The minimizers of the above functional have measure zero boundary in nice metric measure spaces, and so in these spaces they have an open representative. However, this open representative need not be a \(BV\)-extension set (see \cite[Example 4.1]{KoivuLucicRajala2023}).
As we will see in Proposition \ref{prop_closedrepr}, it is generally true that if a Borel set $A$ is an extension set and if it has a closed representative, then this representative is also an extension set. This observation leads us to a natural question: does the closed representative have better (extension) properties than the original set? A more explicit question that we will work towards in this paper is the following.

\begin{question}\label{q:main}
Let $E \subset \X$ be a closed $BV$-extension set. Under what assumptions on $E$ and on the metric measure space $(\X,\sfd,\mm)$ is $E$ also a $W^{1,1}$-extension set?
\end{question}

For a closed set $E \subset \X$ let us consider the following claims; we refer the reader to Section \ref{sec:preliminaries} for the relevant definitions:
\begin{itemize}
    \item [(\( BV \))] \(E\) has the \(BV\)-extension property.
    \item[(s-$BV$)] $E$ has the strong $BV$-extension property.
    \item[($W^{1,1}$)] $E$ has the $W^{1,1}$-extension property.
    \item[($W_w^{1,1}$)] $E$ has the $W_w^{1,1}$-extension property.
\end{itemize}
In Proposition \ref{prop:general_implications} we prove the following 
implications:
\[
\begin{array}{ccc}
( \text{s-}BV ) & \Longrightarrow & ( W^{1,1} ) \\
\Downarrow && \Downarrow\\
(W_w^{1,1}) &\Longrightarrow & ( BV )
\end{array}
\]
Coming back to the problem studied in this paper, Question \ref{q:main} asks when a partial converse to Proposition \ref{prop:general_implications} holds. First of all, we provide an example (see Example \ref{ex:BVnotW11}) showing the existence of a closed \(BV\)-extension set, which is not a \(W^{1,1}\)- nor a \(W^{1,1}_w\)-extension set. Thus Question \ref{q:main} is indeed meaningful, and one has to identify some further properties of the set or of the space in order to get a positive answer. 
On the one hand, the space in the Example \ref{ex:BVnotW11} does not satisfy the PI-assumption. On the other hand, we are able to answer Question \ref{q:main} positively in the Euclidean planar case, for closed sets with finitely many connected components in the complement. Thus, this leaves an open question for further investigations in this direction:

\begin{question}
Are closed $BV$-extension sets in $\mathbb R^n$, with $n \ge 2$, $W^{1,1}$-extension sets?
If so, is the same true also in PI spaces?
\end{question}

After the study of the relations between extendability properties with respect to different classes of functions in the full generality in Section \ref{sec:general}, in order to pass from the Example \ref{ex:BVnotW11} to a positive result about the problem \(BV\)-extension implies \(W^{1,1}\)-extension, we obtain several result in the setting of PI spaces that are of independent interest; we present them next.

\subsection*{Geometry of $BV$- and $\ca{BV}$-extension sets in PI spaces}

A preliminary analysis to tackle Question \ref{q:main} is to study geometric and analytic properties of $\ca{BV}$- and $BV$-extension sets (the \(\ca{BV}\) standing for the homogeneous \(BV\) space), under the additional assumption that the metric measure space is doubling and satisfies a $(1,1)$-Poincar\'{e} inequality.

We first prove that closed $BV$-extension sets satisfies the measure density condition. Namely, there exists a constant $C > 0$ so that for all $x \in E$ and $0 < r < {\rm diam}(E)$ we have
\[
\mm(B(x,r)\cap E) \ge C \mm(B(x,r)).
\]
See Proposition \ref{prop:measdens} for the precise statement and the proof. The measure density fails if $\X$ is not a PI space (see Example \ref{ex:sierpinski}).

Second, we relate $\ca{BV}$-extension sets to $BV$-extension sets. We prove that in PI spaces bounded $\ca{BV}$-extension sets are $BV$-extension sets (Proposition \ref{prop:homBVtofullBV}).
The result does not hold if we drop the PI-assumption, as Example \ref{example:earring} shows.

Then, we analyze decomposability properties of closed $BV$- and $\ca{BV}$- extension sets. We prove that closed $\ca{BV}$-extension sets are indecomposable (no PI-assumption is needed here), while this is not the case for $BV$-extension sets as simple examples show.
However, in the case of compact $BV$-extension sets we show (see Lemma \ref{lma:connected_full}) that there exists a finite decomposition of $E$ into measurable sets $\{E_i\}_{i=1}^m$ with $m \in \mathbb{N}$, such that
\[\mm(E)=\sum_{i=1}^m \mm(E_i)\qquad\text{ and }\qquad
\Pe_E(E_i,E)=0\, \text{ for every }i=1, \ldots,m.
\]
and $E_i$ are indecomposable in $E$. See Definition \ref{def:decomposability} for the definition of decomposability.
\smallskip

We conclude this introduction commenting on the proof strategy for the positive result we achieve in the Euclidean planar case.

\subsection*{$BV$-extension sets in the Euclidean plane}

Our main theorem states the following. 
\begin{theorem}\label{thm:main}
 Let $E \subset \mathbb R^2$ be a compact $BV$-extension set. If $\mathbb R^2 \setminus E$ has finitely many connected components, then $E$ is a $W^{1,1}$-extension set.
\end{theorem}

The proof of this result is tailored for the Euclidean plane and combines some well-known and some new analytic and geometric results related to $BV$-extension sets and to sets of finite perimeter in $\mathbb{R}^2$. 
The geometry of the essential boundary of sets of finite perimeter in Euclidean spaces has been studied by 
 Ambrosio et al. in \cite{ALC01}. In particular, by \cite[Corollary 1]{ALC01} we have that the essential boundary of a set of finite perimeter in the plane can be written as a countable union of Jordan loops.

 Given a Jordan domain $\Omega$, we can modify a curve whose image is in $\overline{\Omega}$ to a curve whose image, beside the endpoints, is in $\Omega$. This can be done by increasing the length of the curve in a controlled way. This was proved by Garcia-Bravo and the fourth-named author in \cite[Lemma 5.3]{BR21}, employing the quasiconvexity property of the connected components of the complement of a \(BV\)-extension domain in consideration. The latter property fails in the case of closed \(BV\)-extension sets.

To rely on the strategy from \cite{BR21}, we study geometric and analytic properties of the complement of $BV$-extension closed sets in the plane, having the following result.
    \begin{proposition}[Lemma {\ref{lma:quasiconvexity_full}}]
    If $E \subset \mathbb R^2$ is a compact $BV$-extension set, then 
    there exist constants $\delta,C>0$ so that the following four properties hold for all connected components $\Omega$ of $\mathbb{R}^2 \setminus E$ and all $x,y \in \overline{\Omega}$.
    \begin{enumerate}
     \item $\sfd_{\Omega}(x,y) \le C + \|x-y\|$.
     \item If $\sfd_{\Omega}(x,y)< \delta$, then $\sfd_{\Omega}(x,y) \le C\|x-y\|$.
    \item $(\partial \Omega,\sfd_\Omega)$ is totally bounded. 
    \item For $x \in \Omega$ and $0 < r < \delta/2$ the closed ball $\overline{B}_{\sfd_\Omega}(x,r)$ is a closed subset of $\mathbb R^2$.
    \end{enumerate}
    \end{proposition}
    
    The distance $\sfd_\Omega$, restricted to $\Omega$, is the length distance on $(\Omega,\sfd)$. For the definitions, see Section \ref{sec:preliminaries} for the precise definitions. Property (1) states that $(\bar{\Omega}, \sfd)$ is quasi-isometric to $(\bar{\Omega}, \sfd_\Omega)$, while property (2) states a local (with respect to $\sfd_\Omega$) quasiconvexity of $\Omega$. 

 We can use the previous step to modify a set of finite perimeter in every connected component $\Omega$ of the complement of a compact $BV$-extension set (see Lemma \ref{lem:kicking}). The modification is done in such a way the set does not change outside $\bar{\Omega}$, the perimeter increases in a controlled way and the perimeter of the modified set on $\partial \Omega$ is zero.

 Finally, the proof of Theorem \ref{thm:main} is structured as follows. Given a $BV$-extension set, we can extend a function $u\in W^{1,1}(E)$ to $\tilde{u} \in BV(\mathbb{R}^2)$. We modify this function in a controlled way in every connected component $\Omega_i$ of the complement of $E$. The modification is done in such a way that the new function is in $W^{1,1}(\Omega_i)$ and does not jump over the boundary.

 To do this, we need as a preliminary reduction during the proof to replace $E$ with the closure of points of density one $\overline{E^1}$. The $\overline{E^1}$ is a closed $BV$-extension set, if $E$ is and the connected components $\{\Omega_i^1\}$ of the complement have small intersection in a measure-theoretic sense, namely

 \begin{equation*}
    \partial \Omega_i^1 \cap \partial \Omega_j^1 \text{ is finite for }i \neq j.
 \end{equation*}
 See Proposition \ref{lma:small_intersect} for this result.

\section*{Acknowledgments}

The first named author is supported by the European Union’s Horizon 2020 research and innovation programme (Grant
agreement No. 948021). The third named author is supported by the Research Council of Finland, grant no.\ 362689.

\section{Preliminaries}
\label{sec:preliminaries}
\paragraph{\bf General notation}
In this paper, $(\X,\sfd,\mm)$ is a metric measure space, i.e.\ $(\X,\sfd)$ is a complete and separable metric space and $\mm$ is a nonnegative Borel measure that is finite on bounded sets.
\smallskip

We denote by $\mathscr{B}(\X)$ the Borel subsets of $\X$. We denote by ${\rm Lip}(\X)$, ${\rm Lip}_{\rm loc}(\X)$ respectively the space of Lipschitz functions and locally Lipschitz functions from $\X$ to $\mathbb{R}$. We denote by $C([0,1],\X)$ and ${\rm Lip}([0,1],\X)$ respectively the space of continuous and Lipschitz curves with values in $\X$. We endow $C([0,1],\X)$ with the supremum norm, that makes it a complete metric space. Given two metric spaces $(\X,\sfd_\X)$ and $(\Y,\sfd_\Y)$ and a Lipschitz function $u \colon \X \to \Y$ we denote by ${\rm Lip}(u)$ its Lipschitz constant.
Given a function $u \colon \X \to \Y$, we define the local Lipschitz constant $\lip\, u \colon \X \to \R$ as
\begin{equation*}
    \lip \,u(x):=\begin{cases}
    \lims_{y \to x} \frac{\sfd_\Y(u(y),u(x))}{\sfd_\X(x,y)},&\text{if }x\text{ is an accumulation point,}\\
    0,&\text{otherwise.}
    \end{cases}
\end{equation*}
Moreover, given $x\in \X$ and $r >0$, we define ${\rm Lip}(u,B(x,r)):=\sup_{ z\neq w,\, z,w \in B(x,r)} \frac{\sfd_\Y(u(z),u(w))}{\sfd_\X(z,w)}$, with the convention ${\rm sup}\,\emptyset=0$. The asymptotic Lipschitz constant $\lip_a u \colon \X \to [0,\infty]$ is defined as $\lip_a u(x):=\lim_{r\to 0} {\rm Lip}(u,B(x,r))=\inf_{r\to 0} {\rm Lip}(u,B(x,r))$ for $x\in \X$.
\smallskip

We sometimes denote the image of a curve $\gamma \colon [a,b] \to \X$ by $\gamma$ instead of $\gamma([a,b])$. We denote the length of a curve $\gamma$ by $\ell(\gamma)$. Given $x,y \in \R^n$, we use the notation $[x,y]$ for the line-segment from $x$ to $y$, that is, a curve $[0,1] \to \R^n \colon t \mapsto (1-t)x + ty$.
We use the following notation for internal distance.
Let $O \subset \X$ be open. Then 
we define the function $\sfd_O \colon \X \times \X \to [0,\infty]$ as
\[
\sfd_O(x,y) = \inf\left\{\ell(\gamma)\,:\, \gamma \text{ is a curve joining }x \text{ and }y\text{ in }O\cup\{x,y\}\right\}\quad \text{ for every }x,y\in \X,
\]
with the convention that $\sfd_O(x,y)=+\infty$ if the infimum is taken over the empty set. The function $\sfd_O\colon\X \times \X \to \mathbb{R}^2$ does not satisfy in general the axioms of a distance. For instance, in the case of $O\subset \mathbb{R}^2$, defined as $O:=B((0,0),1)\cup B((2,0),1)$, the function $\sfd_O$ does not satisfy the triangular inequality. However, using that $\sfd(x,y)\le\sfd_O(x,y)$, we have that $\sfd_O \colon O\times O \to [0,\infty]$ satisfies the axioms of a distance taking possibly infinite values. If $O$ is rectifiable and path-connected, we have that $\sfd_O$ is a distance taking finite values. Given $x \in \X$ and $r>0$, we set
\begin{equation*}
\begin{aligned}
    & B_{\sfd_O}(x,r) :=\{y \in \X : \sfd_O(x,y) < r\},\\
    & \overline{B}_{\sfd_O}(x,r):=\{y \in \X : \sfd_O(x,y) \le r\}.\\
\end{aligned}
\end{equation*}
As simple examples show, in general $B_{\sfd_O}(x,r)$ is neither open, nor closed.

Given a set \(E\subset \X\) we denote by \(E^0\) and \(E^1\) the set of points where $E$ has respectively density \(0\) and \(1\) (with respect to the measure \(\mm\)), namely
\begin{equation*}
    E^0:=\left\{ x\in \X: \, \lims_{r \to 0} \frac{\mm(B(x,r) \cap E)}{\mm(B(x,r))} =0 \right\},\quad E^1:=\left\{ x\in \X: \, \limi_{r \to 0} \frac{\mm(B(x,r) \cap E)}{\mm(B(x,r))} =1 \right\}.
\end{equation*}
We also recall that the measure-theoretic boundary of a set \(E\subset \X\) is defined as
\[
\partial^M E\coloneqq \left\{x\in \X:\, \lims_{r\to 0}\frac{\mm(E\cap B(x,r))}{\mm(B(x,r))}>0\, \text{ and }\, \lims_{r\to 0}\frac{\mm((\X\setminus E)\cap B(x,r))}{\mm(B(x,r))}>0\right\}.
\]

We denote by $L^0(\mm)$ the space of $\mm$-measurable functions $f\colon \X \to \R$ up to the quotient given by $\mm$-a.e.\ equality. For every \(p\in [1,\infty)\) we denote by \(L^p(\mm)\) the subspace of \(L^0(\mm)\) consisting of \(p\)-integrable elements of \(L^0(\mm)\), and \(L^\infty(\mm)\) the space consisting of all essentially-bounded ones. The notation \(L^p_{\rm loc}(\mm)\) stands for the set of locally \(p\)-integrable functions. Whenever the  measure on the space is understood from the context, we might also use the notation \(L^p(E)\) instead of \(L^p(\mm\restr{E})\), for \(E\subset\X\) (and similarly for the space of locally integrable functions).

\smallskip

\paragraph{\bf Spaces \(BV\) and \(W^{1,1}\)} 
\begin{definition}[Total variation]
Let $(\X,\sfd,\mm)$ be a metric measure space. Consider $u \in L^1_{\rm loc}(\X)$. Given an open set $A \subset \X$, we define
\begin{equation*}
    |D u|(A):= \inf \left\{ \limi_n \int_A \lip\, u_n\, \d \mm:\, u_n \in {\rm Lip}_{\rm loc}(A),\, u_n \to u \in L^1_{\rm loc}(\mm \restr{A}) \right\}. 
\end{equation*}
\end{definition}
We extend $|D u|$ to all Borel sets $B \in \mathscr{B}(\X)$ by setting
\[ |D u|(B):= \inf \left\{ |D u|(A), B \subset A, A \text{ is an open set}\right\}. \]
With this construction, $|D u| \colon \mathscr{B}(\X) \to [0,\infty)$ is a Borel measure, called the \emph{total variation measure} of $u$ (\cite[Thm.\ 3.4]{Mir03}).
It follows from the definition of total variation that, given an open set $A \subset \X$, if
$u_n \to u$ in $L^1_{\textrm{loc}}(A)$ then $|D u|(A) \le \limi_{n \to \infty} |D u_n|(A)$.
\smallskip

Given a $B \in \mathscr{B}(\X)$ and $u \in L^1_{\textrm{loc}}(B)$, we denote by 
\[|D u|_B \coloneqq \text{ the total variation of $u$ computed in the metric measure space $(\X,\sfd,\mm\restr{B})$.}
\]
If the set \(B\) is closed, then from \cite[Corollary  2.5]{KoivuLucicRajala2023}, we have that 
\begin{equation}
\label{eq:from_B_to_X}
|D u|_B(B) = \inf \left\{ \liminf_{n\to \infty} \int_B \lip_a (u_n)\, \d \mm:\, u_n \in {\rm LIP}(\X),\, u_n \to u \in L^1_{\rm loc}(\mm \restr{B}) \right\} = |D u|_B(\X). 
\end{equation}
\begin{definition}[The spaces $\ca{BV}(B)$ and $BV(B)$]\label{def:BV_on_borel}
Let $(\X,\sfd,\mm)$ be a metric measure space. Let $B \subset \X$ be Borel. 
We define
\[
\begin{split}
\ca{BV}(B):=&\big\{ u\in L^1_{\loc}(\mm\restr{B}):\,|D u|_B(B)<+\infty \big\},\\
BV(B):=&\big\{ u \in L^1(\mm\restr{B}):\,|D u|_B(B)<+\infty\big\}.
\end{split}
\]
We endow the space \(\ca{BV}(B)\) 
with the seminorm
 and the space 
\(BV(B)\) with the norm given by 
\[\|u\|_{\ca{BV}(B)}:= |D u|_B(B)\quad\text{and}\quad
\|u\|_{BV(B)}:=\|u\|_{L^1(\mm\restr{B})}+|D u|_B(B),
\]
respectively.
\end{definition}
Observe, in particular, that for a closed set \(E\subset\X\) it holds \(BV((E,\sfd\restr{E\times E}, \mm\restr{E}))=BV(\X,\sfd, \mm\restr{E})\).
We recall the following notion from \cite[Definition 2.11]{CKR23}.

\begin{definition}[Sets of finite perimeter on a Borel subset B]
    Let $(\X,\sfd,\mm)$ be a metric measure space and let $B\subset \X$ be Borel. We say that $E \in \mathscr{B}(B)$ has finite perimeter on $B$ if $\Pe_{B}(E) <\infty$, where $\Pe_{B}(E) := |D \chi_E|_B(B)$.
    Moreover, we define for every Borel set $F$, $\Pe_B(E,F):= |D \chi_E|_B(B \cap F)$.
\end{definition}

Next, we recall the definitions of subsets of \(\ca{BV}(\X)\) or \(BV(\X)\) that will be relevant throughout. 

\begin{definition}
 Let \((\X,\sfd, \mm)\) be a metric measure space. We define 
 \[
 \begin{split}
L^{1,1}_w(\X)\coloneqq \{u\in \ca{BV}(\X):\, |Du|\ll \mm\},\\
W^{1,1}_w(\X)\coloneqq \{u\in BV(\X):\, |Du|\ll \mm\}.
 \end{split}
 \]
 We endow these spaces with the seminorm \(\|\cdot\|_{L^{1,1}_w(\X)}\) and the norm \(\|\cdot\|_{W^{1,1}_w(\X)}\) given respectively by
 \[\|u\|_{L^{1,1}_w(\X)}:=\left\|\frac{\d |D u|}{\d \mm} \right\|_{L^1(\mm)}\quad \text{ and }\quad \| u\|_{W^{1,1}_w(\X)}:=\| u\|_{L^1(\mm)}+\| u \|_{L^{1,1}_w(\X)}.\]
\end{definition}
In the case of a closed set \(E\subset \X\), we define the spaces 
 \[L^{1,1}_w(E)\coloneqq L^{1,1}_w(E,\sfd \restr{E\times E}, \mm\restr{E})\quad \text{ and }\quad  W^{1,1}_w(E)\coloneqq W^{1,1}_w(E,\sfd\restr{E\times E}, \mm\restr{E}).\] 
There are several possible definitions of $W^{1,1}$ (and homogenous versions) in metric measure spaces, 
see for instance \cite{Ambrosio-DiMarino14}. 
Our main working definition is the following one, defined in terms of a curvewise approach. 
This approach has been recently shown to be equivalent to several others (among which the Newtonian-Sobolev space \(N^{1,1}(\X)\), or the approach via approximations by Lipschitz functions leading to the space \(H^{1,1}(\X)\)); see \cite{AILP24}.
To introduce it, we recall the notion of an $\infty$-test plan. 
We define the evaluation map at time $t \in [0,1]$ as the map $\e_t\colon C([0,1],\X) \to \X$ defined as $\e_t (\gamma):=\gamma_t$. It follows from the definition that $\e_t$ is $1$-Lipschitz (hence Borel).
Given a Borel set \(B\subseteq \X\), we say that $\ppi \in \mathscr{P}(C([0,1],\X))$ is an $\infty$-test plan on \(B\) if
\begin{itemize}
    \item $\ppi$ is concentrated on ${\rm Lip}([0,1],B)$ and $(\gamma \mapsto {\rm Lip}(\gamma))\in L^\infty(\ppi)$;
    \item there exists $C=C(\ppi)$ such that ${\e_t}_*\ppi \le C \mm$ for every $t \in [0,1]$.
\end{itemize}


\begin{definition}[The spaces $L^{1,1}(\X)$ and $W^{1,1}(\X)$]\label{def:W11}
Let \((\X,\sfd, \mm)\) be a metric measure space. Given $u \in L^0(\mm)$, we say that $u \in L^{1,1}(\X)$ if there exists $G \in L^1(\mm)$ such that, for every $\infty$-test plan $\ppi$
\begin{equation}\label{eq:1_wug}
|u(\gamma_1) - u(\gamma_0) | \le \int_0^1 G(\gamma_t)|\dot{\gamma_t}|\,\d t\qquad\text{for }\ppi\text{-a.e. }\gamma.
\end{equation}
In this case, we call $G$ a $1$-weak upper gradient of $u$. The minimal (in the $\mm$-a.e.\ sense) $1$-weak upper gradient is denoted by $|Du|_{1,\X}$ (its existence can be proved by arguing as in \cite[Section 2.2]{CKR23}). We endow $L^{1,1}(\X)$ with the seminorm $\|u\|_{L^{1,1}}\coloneqq\||D u|_{1,\X}\|_{L^1(\mm)}$.
We define the Sobolev space $W^{1,1}(\X, \sfd, \mm)$ as $W^{1,1}(\X, \sfd, \mm):=L^1(\mm)\cap L^{1,1}(\X)$.
The norm in $W^{1,1}(\X, \sfd,\mm)$ is then given by
\[
    \| u \|_{W^{1,1}(\X,\sfd,\mm)}:= \| u \|_{L^{1}(\mm)}+ \| |D u|_{1,\X} \|_{L^{1}(\mm)}.
\]
\end{definition} 
%
As before, it is easy to extend such definitions to a closed set $E \subset \X$. Indeed, in this case, we define 
$L^{1,1}(E)\coloneqq L^{1,1}(E,\sfd\restr{E\times E}, \mm\restr{E})$ and $W^{1,1}(E)\coloneqq W^{1,1}(E,\sfd\restr{E\times E}, \mm\restr{E})$.
%
%
\smallskip

In the case of an open subset \(\Omega\subset \X\) we consider the following definition.

\begin{definition}[The space \(W^{1,1}(\Omega)\)]
Let \((\X,\sfd,\mm)\) be a metric measure space and let \(\Omega\subset \X\) be an open set. The Sobolev space \(W^{1,1}(\Omega)\) is defined as the space of those functions \(u\in L^1(\mm)\) for which the 1-weak upper gradient inequality \eqref{eq:1_wug} holds for every \(\infty\)-test plan \(\ppi\) on \(\Omega\). The corresponding minimal 1-weak upper gradient of \(u\in W^{1,1}(\Omega)\) is denoted by \(|Du|_{1,\Omega}\).
\end{definition}
One of the main 
tools developed in \cite{CKR23} is the existence of a smoothing operator on an open set from $BV$ to ${\rm Lip}_{\rm{loc}}$, that we will use in this work.

\begin{proposition}[Smoothing operator {\cite[Proposition 4.1]{CKR23}}]
\label{prop:smoothing}
Let \((\X, \sfd, \mm)\) be a metric measure space.
Let $\Omega \subset \X$ be open. There exists a constant $C$ such that the following holds: for every $\varepsilon >0$, there exists $T_\varepsilon \colon BV(\Omega) \to {\rm Lip}_{\rm{loc}}(\Omega)$ such that for every \(u\in BV(\Omega)\) it holds that 
\begin{equation}
\label{eq:smoothing_normbounds}
\|T_\varepsilon u\|_{L^1(\Omega)} \le \|u\|_{L^1(\Omega)} + \varepsilon,\quad
\int_{\Omega} \lip\,T_\varepsilon u\,\d \mm \le C(|D u|(\Omega) +\varepsilon),
\end{equation}
$T_\varepsilon u-u \in BV(\X)$ (when defined to be $0$ in $\X \setminus \Omega$) and
\begin{equation}
\label{eq:notchargingtheboundary}
    |D(T_\varepsilon u - u )|(\partial \Omega) = 0.
\end{equation}
\end{proposition}
An important property of such a smoothing operator is \eqref{eq:notchargingtheboundary}, which states that the difference $(T_\varepsilon u-u)\chi_\Omega$ (as a function defined on $\X$) does not have total variation on the topological boundary.
Another lemma we used in \cite{CKR23} gives a sufficient condition on how to glue a Sobolev function on an open set and another one defined on the complement to get a globally defined one.
\begin{lemma}[{\cite[Lemma 4.4]{CKR23}}]
\label{lemma:fromlocal_to_global_W11}
Let \((\X,\sfd,\mm)\) be a metric measure space and \(\Omega\subset\X\) be a nonempty open set.
Let further \(u\in BV(\X)\) be such that 
\(\chi_\Omega u= u_1\) for some \(u_1\in W^{1,1}(\Omega)\) and \(\chi_{\X\setminus \bar \Omega}u=u_2\) some \(u_2\in W^{1,1}(\X\setminus \bar\Omega)\). Assume that \(|Du|(\partial \Omega)=0\).
Then \(u\in W^{1,1}(\X)\) and it holds that 
\[
|Du|_{1,\X}\leq \chi_{\Omega}|Du_1|_{1,\Omega}+\chi_{\X\setminus \bar\Omega} |Du_2|_{1,\X\setminus \bar\Omega}\quad \mm\text{-a.e.\ in }\X.
\] 
\end{lemma}
\begin{remark}\label{rmk:lem44}
We remark that whenever in the above lemma we have that 
\(\X\setminus\bar\Omega=\emptyset\)
and \(\chi_\Omega u=u_1\) for some \(u_1\in W^{1,1}(\Omega)\) with \(|Du|(\partial \Omega)=0\), then \(u\in W^{1,1}(\X)\) and 
\[
|Du|_{1,\X} =\chi_{\Omega}|Du|_{1,\X}\leq \chi_{\Omega}|Du_1|_{1,\Omega}\quad \mm\text{-a.e.\ in }\X.
\]
\end{remark}

\subsection{PI spaces}
We say that $\mm$ is doubling, if there exists a constant $C_\mm>0$ such that for every $x \in \X$ and $r>0$
we have
\[
0 < \mm(B(x,2r)) \le C_\mm \mm(B(x,r)) < \infty.
\]

Recall \cite[Lemma 14.6]{HK2000} that the doubling assumption on $\mm$ implies the existence of a constant $C>0$ so that we have
\begin{equation}\label{eq:s_doubling}
\frac{\mm(B(x,r))}{\mm(B(x_0,R))} \ge C\left(\frac{r}{R}\right)^s,\quad \text{ with } s \coloneqq \log_2 C_\mm,
\end{equation}
for every $x_0\in \X$, $x \in B(x_0,R)$ and $0 < r \le R$. Without loss of generality, we may assume that $s >1$ in \eqref{eq:s_doubling}.

We say that a metric measure space $(\X,\sfd,\mm)$ satisfies the $(1,1)$-Poincar\'e inequality, if there exist constants $C>0$ and $\lambda \ge 1$ such that for every $u \in {\rm Lip}_{\rm loc}(\X)$, every $x \in \X$ and $r>0$ we have
\[
 \dashint_{B(x,r)}|u-u_{B(x,r)}| \,\d\mm\le Cr\dashint_{B(x,\lambda r)}\lip\,u\,\d\mm,
\]
where $\dashint$ denotes the average integral and $u_A$ the average of $u$ in a set $A$.

\begin{remark}
By the definition of $|D u|$ for $u \in BV(\X)$ in \cite{Mir03}, we have that $(\X,\sfd,\mm)$ satisfies a $(1,1)$-Poincar\'{e} inequality if and only if
there exist constants $C>0$ and $\lambda \ge 1$ such that for every $u \in BV(\X)$, every $x \in \X$ and $r>0$ we have
    \begin{equation}
    \label{eq:poincare_BV_rem4.1}
        \dashint_{B(x,r)}|u-u_{B(x,r)}| \,\d\mm\le Cr\frac{|Du|(B(x,\lambda r))}{\mm(B(x,\lambda r))}.    
    \end{equation}
\end{remark}

Whenever in this note we say that $(\X,\sfd,\mm)$ is a PI space, we mean that \((\X,\sfd, \mm)\) is a metric measure space with the measure $\mm$ doubling and the space satisfies the $(1,1)$-Poincar\'e inequality.
From \cite[Theorem 5.1]{HK2000} we have the following.
\begin{proposition}\label{prop:Sobo}
Let $(\X,\sfd,\mm)$ be a PI space. Then there exist constants $C>0$ and \(\lambda\geq 1\) so that, being \(s\) as in \eqref{eq:s_doubling}, for every $u \in {\rm Lip}_{\rm loc}(\X)$, $x \in \X$, $r>0$, and $t>0$ we have
\[
\frac{\mm(\{z \in B(x,r)\,:\,|u(z) - u_{B(x,r)}|>t\})t^{\frac{s}{s-1}}}{\mm(B(x,r))} \le C\left(r\dashint_{B(x,5\lambda r)}\lip_a\,u\,\d\mm\right)^{\frac{s}{s-1}}.
\]
Moreover, there exists a constant $C>0$ such that for every $u \in BV(\X)$, $x\in \X$, $r>0$, and $t>0$
\[
\frac{\mm(\{z \in B(x,r)\,:\,|u(z) - u_{B(x,r)}|>t\})t^{\frac{s}{s-1}}}{\mm(B(x,r))} \le C\left(r \frac{|D u|(B(x,5\lambda r))}{\mm(B(x,5\lambda r))}\right)^{\frac{s}{s-1}}.
\]
\end{proposition}


\section{Extension sets in general metric measure spaces}\label{sec:general}

In this section we will present relations between different extension properties of sets that hold in general metric measure spaces. 

\begin{definition}[Extension property for $(\mathscr{A},\|\cdot\|_{\mathscr{A}})$]
\label{def:extensions}
Let $E\subset \X$ be a Borel set.
For every Borel set $U\subset \X$, let $\mathscr{A}(U)$ be a vector subspace of $L^0(\mm\restr{U})$ and $\|\cdot\|_{\mathscr{A}(U)}$ be a seminorm on it. We say that $E\subset \X$ has the extension property for $\mathscr{A}$ with respect to $\|\cdot\|_{\mathscr{A}}$ (or is a $(\mathscr{A},\|\cdot\|_{\mathscr{A}})$) extension set) if there exists a constant $C \ge 1$ and $T \colon \mathscr{A}(E) \to \mathscr{A}(\X)$ such that 
\begin{itemize}
    \item $T u =u$ on $E$ for every $u \in \mathscr{A}(E)$;
    \item $\|T u \|_{\mathscr{A}(\X)} \le C \| u \|_{\mathscr{A}(E)} $ for every $u \in \mathscr{A}(E)$.
\end{itemize}
Sometimes, in order to keep track of the function spaces associated with the operator \(T\), we will write \(T_{\mathscr A(E)\to \mathscr A(\X)}\).
\end{definition}

In this case, we define 
\begin{equation}
    \| T \|_{\mathscr{A}}:= \sup_{u \in \mathscr{A}(E)} \frac{\|T u\|_{\mathscr{A}(\X)}}{\| u \|_{\mathscr{A}(E)}} <\infty
\end{equation}
with the convention that $0/0=0$ and $t/0 = \infty$ for $t>0$.

Notice that we do not assume the extension operators to be linear. We consider various extension properties. We say that a set $E \subset \X$ satisfies ($\ca{BV}$) if $E$ has the extension property for $\ca{BV}$-functions with respect to $\|\cdot\|_{\ca{BV}}$. We define the properties ($BV$), ($L^{1,1}$), ($W^{1,1}$), ($L_w^{1,1}$), and ($W_w^{1,1}$) similarly. 


Although for closed extension sets they do not play such an important role, for completeness, we also consider the strong $\ca{BV}$- and $BV$-extension properties.
A set $E \subset \X$ satisfies (s-$\ca{BV}$) if it has the \emph{strong $\ca{BV}$-extension property} meaning that there is a constant $C>0$ so that for every $u \in \ca{BV}(E)$ there exists $Tu \in \ca{BV}(\X)$ with $\|Tu\|_{\ca{BV}(\X)}\le C\|u\|_{\ca{BV}(E)}$, $Tu\restr{E} = u$, and $|D(Tu)|(\partial E) = 0$.
We define (s-$BV$) similarly.


In the paper we focus on extension properties for closed sets. Typically, the extension properties improve when we go to the closure of a set. The more precise statement is as follows.

\begin{proposition}\label{prop_closedrepr}
Let $A \subset \X$ be a measurable set and $B \subset \X$ a closed set such that $\mm(B \triangle A)=0$. Then any of the properties 
($\ca{BV}$), ($BV$), ($L^{1,1}$), ($W^{1,1}$), ($L_w^{1,1}$), ($W_w^{1,1}$), (s-$\ca{BV}$), (s-$BV$)  
for $A$ implies the same property for $B$.
\end{proposition}
\begin{proof}
Let us start with the proof for ($\ca{BV}$) and ($BV$).
Let $u \in BV(B)$. Since $B$ is closed, \eqref{eq:from_B_to_X} holds true.
Moreover, since $\mm(B \triangle A)=0$, we have $|D u|_A(\X) = |D u|_B(\X)$. Hence, $u \in BV(A)$ and 
\[
|Du|_A(A) \le |Du|_A(\X)= |D u|_B(\X)\stackrel{\eqref{eq:from_B_to_X}}{=} |Du|_B(B).
\]
Consequently, for the extension operator $T \colon BV(A) \to BV(X)$ we have
\[
\|T\|_{BV(B) \to BV(X)} \le \|T\|_{ BV(A) \to BV(X)}
\]
for both, the homogeneous and full norm. The claim for ($L_w^{1,1}$) and ($W_w^{1,1}$) follows directly from the above.

Let us then prove the claim for  ($L^{1,1}$) and ($W^{1,1}$).  It suffices to show that any $\infty$-test plan $\ppi$ on $A$ is also an $\infty$-test plan on $B$. Suppose then that $\ppi$ is not an $\infty$-test plan on $B$. Then $\ppi$ is not concentrated on $\Lip([0,1],B)$. Thus there exists $\Gamma \subset \Lip([0,1],A)$ with $\ppi(\Gamma)>0$ such that for every $\gamma \in \Gamma$ there exists $t \in [0,1]$ so that $\gamma_t \notin B$. Since $B$ is closed, for the above $\gamma$ and $t$ there exists $\varepsilon>0$ so that 
\[
\gamma([0,1]\cap [t-\varepsilon,t+\varepsilon]) \cap B \ne \emptyset.
\]
Consequently, there exists $t_0 \in [0,1]$ and $\Gamma'\subset \Gamma$ with $\ppi(\Gamma')>0$ such that for all $\gamma \in \Gamma'$ we have $\gamma_{t_0} \notin B$.

Since $\ppi$ is an $\infty$-test plan, ${e_{t_0}}_\ast \ppi \le C\mm$ and so
$\mm(\{\gamma_{t_0}\,:\, \gamma \in \Gamma'\})>0$ which contradicts the assumption $\mm(B \triangle A)=0$.

Let us now consider the properties (s-$\ca{BV}$) and (s-$BV$). Let $T$ be the strong $BV$-extension operator for $A$ and take $u \in BV(B)$. As noted in the beginning of the proof, $|Du|_A(A) \le |Du|_B(B)$. Hence we only need to prove that $|D(Tu)|(\partial B)=0$. By Proposition \ref{prop:smoothing}, we may assume that $Tu\restr{\X\setminus \overline{A}} \in {\rm Lip}_{\rm{loc}}(\X \setminus \overline{A})$.
Since $\mm(A \Delta B) = 0$ and $B$ is closed, we have $\mm(\partial B \setminus \overline{A}) = 0$ and so $|D(Tu)|(\partial B \setminus \overline{A}) = 0$. Since $T$ is a strong $BV$-extension operator for $A$, we have $|D(Tu)|(\partial A)=0$ and it remains to prove that 
$|D(Tu)|(\partial B \cap {\rm int}\,{A}) = 0$. Towards this, consider $x \in \partial B \cap {\rm int}\,{A}$. Using again the assumption $\mm(A \Delta B) = 0$ we get that $x \notin {\rm spt}(\mm)$. Hence also $x \notin {\rm spt}(|D(Tu)|)$.
\end{proof}



%

\begin{remark}
There are two particularly interesting settings where we can apply Proposition \ref{prop_closedrepr}.
The first one is when $A$ is a closed extension set in a metric space $(\X,\sfd)$ with a doubling reference measure $\mm$. Then, by Lebesgue density point theorem, $\mm(A \Delta \overline{A^1}) = 0$. 
Hence, when studying the properties of $A$, by Proposition \ref{prop_closedrepr} we can first consider $\overline{A^1}$ instead (and then return to $A$). See Section \ref{sec:connectivity}.

The second one is when $A$ is an open extension set (with respect to one of the functional spaces we considered) in a PI space $(\X,\sfd,\mm)$. Then $\mm(\partial A)=0$ as a corollary of a measure density for extension sets (see Proposition \ref{prop:measdens} for an example of such result). Now, by Proposition \ref{prop_closedrepr}, the closure $\overline{A}$ is also an extension set (with respect to the same functional spaces), which might have even improved extension properties compared to $A$, see Theorem \ref{thm:main}.
\end{remark}

Let us then consider the extension properties of closed subsets $E \subset \X$. 
Similar conclusions hold for the homogeneous spaces. However, since in this paper we mostly focus on the spaces with the full norms, we prove the proposition for the full norms.
\begin{proposition}\label{prop:general_implications}
Let \((\X, \sfd, \mm)\) be a metric measure space and $E \subset \X$ a nonempty closed set. Then the following implications for \(E\) hold for the full norms
\[
\begin{array}{ccc}
( \text{s-}BV ) & \Longrightarrow & ( W^{1,1} ) \\
\Downarrow && \Downarrow\\
(W_w^{1,1}) &\Longrightarrow & ( BV ).
\end{array}
\]

\end{proposition}
\begin{proof}\ \\
{\sc Proof of} \((\text{s-}BV)\Rightarrow(W^{1,1})(\text{resp.}(W^{1,1}_w))\): 
Assume that \(E\) is an s-\(BV\) extension set, and denote by \(\bar F\colon BV(E)\to BV(\X)\) an associated extension operator. 
Denote \(\Omega\coloneqq \X\setminus E\neq \emptyset\). 
Fix 
\(u\in A(E)\), where \(A(E)\) is either of the spaces \(W^{1,1}(E)\) or \(W^{1,1}_w(E)\).
Fix also \(\varepsilon >0\), and  define
\(u_1\coloneqq T_\varepsilon(\bar F u|_\Omega) \) and \(u_2\coloneqq u\).
Here, 
\(T_\varepsilon\colon BV(\Omega)\to \Lip_{\loc}(\Omega)\) stands for the smoothing operator given by Proposition \ref{prop:smoothing}.
It is then clear that \(u_1\in W^{1,1}(\Omega) \cap W^{1,1}_w(\Omega) \). Also, since
\(\X\setminus \bar\Omega= {\rm int}(E)\subset E\), we have that 
\(u_2|_{\X\setminus \bar \Omega}\in A(\X\setminus \bar \Omega)\).
We now define
\[
Fu\coloneqq \chi_{\Omega} u_1 + \chi_{E} u_2, 
\]
where the function \(u_2\) is understood as extended to \(0\) in \(\X\setminus E\).
We claim that the map \(F\) defined above is an \((A(\X), \|\cdot\|_{A(\X)})\)-extension operator.
First of all, by the definition we have that \(Fu=u\) holds \(\mm\)-a.e.\ on \(E\). 
In order to see that \(Fu\in A(\X)\), let us 
first observe that 
we can write
\(Fu=\bar Fu+\chi_\Omega (u_1-\bar Fu|_\Omega)\in BV(\X)\), where, similarly as above, we consider \(\bar F u|_\Omega\) to be extended to \(0\) in \(\X\setminus \Omega\). Then, 
the strong extension property of \(\bar F\) (which gives \(|D\bar Fu|(\partial E)=0\)) and the property \eqref{eq:notchargingtheboundary} in Proposition \ref{prop:smoothing}, we conclude that \(|DFu|(\partial E)=0\). Therefore, in the case \(A(\X)=W^{1,1}_w(\X)\), the proof is complete. 

In the case \(A(\X)=W^{1,1}(\X)\) we will check that the assumptions of Lemma \ref{lemma:fromlocal_to_global_W11} are satisfied:
\(\mm\)-a.e.\ in \(\Omega\) it holds that \(Fu=u_1\) and \(\mm\)-a.e.\ in \(\X\setminus \bar \Omega\) it holds that \(Fu=u_2\). 
As noticed above, it also holds that \(|DFu|(\partial \Omega)=|DFu|(\partial E)=0\). Thus, Lemma \ref{lemma:fromlocal_to_global_W11} (taking into account Remark \ref{rmk:lem44} in the case \(\X\setminus \bar \Omega=\emptyset\)) yields \(Fu\in W^{1,1}(\X)\) and 
it holds that 
\[
|DFu|_{1,\X}\leq \chi_\Omega |Du_1|_{1,\Omega} +\chi_{\X\setminus \bar \Omega} |Du|_{1,\X\setminus\bar \Omega}. 
\]

It remains to verify the norm bounds for \(\|Fu\|_{A(\X)}\). We will use the similar estimates as in \cite[Proposition 4.6]{CKR23}. First, let us estimate the \(L^1(\mm)\)-norm of \(Fu\). 
Using the expression \(Fu=\bar Fu+\chi_\Omega (u_1-\bar Fu|_\Omega)\), the first inequality in \eqref{eq:smoothing_normbounds}, the fact that \(\bar F\) is a strong BV-extension operator (denoting its operator norm by \(C_{\bar F}\)) and that \(A(E)\subset BV(E)\) with \(\|f\|_{BV(E)}\leq \|f\|_{A(E)}\) for all \(f\in A(E)\), we get
\[
\begin{split}
\|Fu\|_{L^1(\mm)}\leq &\, \|\bar Fu\|_{L^1(\mm)}+\|T_\varepsilon(\bar Fu\restr{\Omega})\|_{L^1(\Omega)} \\
\leq &\, \|\bar Fu\|_{L^1(\mm)}
+ \|\bar Fu\restr{\Omega}\|_{L^1(\Omega)}+\varepsilon
\leq 2\|\bar F u\|_{L^1(\mm)}+\varepsilon \\
\leq &\, 2 C_{\bar F}\|u\|_{BV(E)}+\varepsilon
\leq 2 C_{\bar F} \|u\|_{A(E)} +\varepsilon.
\end{split}
\]
In the case \(A(\X)=W^{1,1}(
\X)\) we need to estimate \(\||DFu|_{1,\X}\|_{L^1(\mm)}\). The estimate for the \(L^1(\mm)\)-norm of \(
\big|\frac{\d |DFu|}{\d \mm}\big|\) (and all the subsequent estimates) in the case \(A(\X)=W^{1,1}_w(\X)\) follows analogously. 
Note first that
\[
\||DFu|_{1,\X}\|_{L^1(\mm)}\leq \big\|\big|D\big(T_\varepsilon(\bar Fu\restr{{\Omega}})\big)\big|_{1,\X}\big\|_{L^1(\Omega)}+\||Du|_{1,X\setminus \bar \Omega}\|_{L^1(\X\setminus \bar \Omega)}.
\]
Let us estimate the first summand in the above inequality. Due to the fact that \(|DT_\varepsilon(\bar F u|_{\Omega})|_{1,\Omega}\leq \lip(T_\varepsilon(\bar Fu|_\Omega)\) holds \(\mm\)-a.e.\ in \(\Omega\) and due to the second inequality in \eqref{eq:smoothing_normbounds}, we get
\[
\begin{split}
\||DT_\varepsilon(\bar F u\restr{{\Omega}})|_{1,\Omega}\|_{L^1(\Omega)}\leq &\, \|\lip(T_\varepsilon(\bar Fu\restr{\Omega)}\|_{L^1(\Omega)}\leq C \big(|D\bar Fu|(\Omega)+\varepsilon\big)\\
\leq &\, C\big(|D\bar F u|(\X) +\varepsilon\big)\leq C\big(\|u\|_{W^{1,1}(E)}+\varepsilon\big).
\end{split}
\]
The second summand \(\||Du|_{1,\X\setminus \bar\Omega}\|_{L^1(\X\setminus \bar \Omega)}\) is bounded above by \(\|u\|_{W^{1,1}(E)}\). All in all, we have that 
\[
\|Fu\|_{W^{1,1}(\X)}=\|Fu\|_{L^1(\mm)}+\||DFu|_{1,\X}\|_{L^1(\mm)}\leq 2C_{\bar F}\|u\|_{W^{1,1}(E)}+(C+1)\|u\|_{W^{1,1}(E)}+ (C+1)\varepsilon.
\]
Choosing \(\varepsilon\) to be precisely \(\|u\|_{W^{1,1}(E)}\), we conclude the proof.




\smallskip

\noindent
{\sc Proof of} \(W^{1,1}(\text{resp. }W^{1,1}_w)\Rightarrow BV\):
 Assume that \(E\) is an \((A(\X),\|\cdot\|_{A(\X)})\)-extension set, where \(A(\X)\) is either of the spaces \(W^{1,1}(\X)\) or \(W^{1,1}_w(\X)\).
To prove the statement, it will be enough for us to show that $E$ is a $(BV \cap L^\infty, \|\cdot\|_{BV})$-extension set. Then, 
by applying verbatim the arguments in step 2) of \cite[Prop.\ 3.4]{CKR23} we can conclude that \(E\) is a \(BV\)-extension set.
  
Thus, let us consider $u \in BV(E) \cap L^\infty(\mm\restr{E})$ and define $\|u\|_{L^\infty(\mm\restr{E})}=:M<\infty$. By definition of $BV(E)$, there exists $(u_i)_i \subset {\rm Lip}_{\rm loc}(E)$ such that $u_i \to u$ in $L^1_{\rm loc}(E)$ and $\int_{E} \lip(u_i)\,\d \mm \to |D u|_E(E)$.
    We can assume without loss of generality that $\|u_i\|_{L^\infty(\mm\restr{E})} \le 2M$. Indeed, 
    by  defining $v_i:=(u_i \wedge 2M) \vee (-2M)$, we get a sequence satisfying 
    \[v_i \to v\,\text{ in} \,L^1_{\textrm{loc}}(E)\quad \text{and } \quad\lip(v_i) \le \lip(u_i)\, \text{ for every }i\in \mathbb N.\]
    Using the fact that $|D u |_E(E)\le \limi_i \int_E \lip(v_i)\,\d \mm$ and the fact that $u_i$ is chosen to be optimal, we get that
    \(\int_E \lip (v_i)\,\d \mm \to |D u|_E(E)\),
    thus proving the claim. We do not relabel the sequence $(u_i)_i$.

    Let us consider $T$ to be the extension operator given by assumption. We define $\tilde u_i:= T u_i \in A(\X)$. Consider $\varphi_M(t):= (t \wedge 2M) \vee (-2M)$. Let us define $\bar{u}_i:=\varphi_M \circ \tilde{u}_i \in A(\X) \cap L^{\infty}(\mm)$. We notice that $\|\bar{u}_i\|_{L^\infty(\mm)}\le 2M$, $\bar{u}_i=u_i$ on $E$, $\| \bar{u}_i\|_{BV(\X)}\le \| \tilde{u}_i\|_{BV(\X)}$ (in particular, we used the fact that the function $\varphi_M$ is $1$-Lipschitz).
    Fix any $p >1$. We set $C:=\sup_i \| \bar{u}_i\|_{L^1(\mm)}$ and observe that $\| \bar{u}_i \|_{L^p(\mm)}^p \le (2M)^{p-1} C$.
    By compactness and the use of Mazur's lemma, we get that there exists a subsequence of \((\bar u_i)_i\) converging in \(L^p(\mm)\) to  some $v\in L^p(\mm)\cap L^\infty(\mm)$. In particular the said subsequence of \((\bar u_i)_i\) converges in $L^1_{\rm loc}(\X)$.
    We do not relabel such sequence as none of the properties above is affected by taking convex combination in Mazur's lemma.   
    Since $u_i \to u$ in $L^1_{\textrm{loc}}(E)$, we get that $v=u$ $\mm$-a.e.\ on $E$.
    The lower semicontinuity of total variation with respect to $L^1_{\textrm{loc}}$ convergence gives
    \begin{equation}
        |D v|(\X) \le \limi_{i \to \infty} |D \bar{u}_i|(\X).
    \end{equation}
    We get
    \begin{equation}
    \begin{aligned}
        \| v\|_{BV(\X)} &\le \limi_{i \to \infty} \|\bar{u}_i\|_{BV(\X)} \le \limi_{i \to \infty} \|\tilde u_i\|_{BV(\X)} \le \limi_{i \to \infty} \|\tilde u_i\|_{A(\X)}\le C \limi_{i \to \infty} \| u_i\|_{A(E)}\\
        & \le C \limi_{i \to \infty} \Big(\|u_i\|_{L^1(\mm\restr{E})} +\int_E \lip(u_i)\,\d \mm\Big) \le C \big( \|u\|_{L^1(\mm|_E)} + |D u|_E(E)\big),
    \end{aligned}
    \end{equation}
    where $C$ is the constant appearing in the definition of $A$-extension set. 
    Therefore, the operator \(T_{BV}\colon BV(E)\cap L^\infty(E)\to BV(\X)\cap L^\infty(\X)\) associating to every \(u\in BV(E)\cap L^\infty(E)\) 
    the function \(v\in BV(\X)\cap L^\infty(\X)\) obtained as above, satisfies all the properties of being an $(BV \cap L^\infty, \|\cdot\|_{BV})$-extension operator, and thus we conclude that \(E\) is a \(BV\cap L^\infty\)-extension set.
\end{proof}
%


%


\begin{remark}
In general, the implications in the above proposition cannot be reversed. We mention several examples in this regard. 
\begin{enumerate}
 \item In Example \ref{ex:BVnotW11} we provide an example of a closed \(BV\)-extension set, which is not a \(W^{1,1}\)- nor a \(W^{1,1}_w\)-extension set. This shows that the Question \ref{q:main} is indeed meaningful, in the sense that one really has to identify some further properties of the set or of the space in order to have the implication holds true. In section \ref{sec:plane} we find that in the Euclidean plane for compact sets with finitely many connected components in the complement the implication in question is verified.
 
 \item In \cite{GarciaBravoRajala2024} it is shown that fat enough planar Sierpi\'nski carpets are $W^{1,1}$-extension sets. However, since $\partial E = E$ for any closed representative of a Sierpi\'nski carpet, there is no representative that would satisfy (s-$BV$) or (s-$\ca{BV}$). 
 

 \item In Example \ref{ex:tripod}
 we show that, even if a closed set \(E\) satisfies \(\mm(\partial E)=0\), being a \(W^{1,1}\)- or a \(W^{1,1}_w\)-extension set does not imply the \(( s\text{-}BV )\)-extension property.
\end{enumerate}
\end{remark}



Towards the Example \ref{ex:BVnotW11}, we first discuss a one-dimensional case, that will be auxiliary in the main construction of the actual example.

%

\begin{example}
\label{example:slice_example}
Consider the space $\X=\mathbb{R}$ endowed with the Euclidean distance and $\mm=\mathcal{L}^1\restr{(-\infty,0]}+ \varepsilon\mathcal{L}^1\restr{(0,\infty)}$. Let $E=[0,\infty)$ and suppose that $T$ is an extension operator from $W^{1,1}(E)$ to $W^{1,1}(\X,\sfd,\mm)$. We bound from below $\|T\|_{W^{1,1}}$. 
We denote for $\Omega \subset \R$ open with $W^{1,1}_e(\Omega)$ the classical Sobolev space in the open set \(\Omega\) of the Euclidean space.

We consider the function $u \in {\rm Lip}_{bs}(E)\subset W^{1,1}(E)$ defined as $u=1$ on $[0,1)$, $u(x)=2-x$ for $x \in [1,2]$ and $u=0$ on $(2,\infty)$. We have that $\|u\|_{W^{1,1}(E)}=\varepsilon \|u\|_{W^{1,1}_e(0,\infty)}=\varepsilon \frac{5}{2}$.
We define $v:=E u \in W^{1,1}(\X)$. In particular, $v \in W^{1,1}_e(\R)$, hence it admists an absolutely continuous representative. We distinguish two cases. 

If $v(x) \ge \frac{1}{2}$ for every $x \in [-1,0]$, we have
\begin{equation*}
    \|v\|_{W^{1,1}(\X)} \ge \|v\|_{L^1(\mm)} \ge \frac{1}{2}.
\end{equation*}

If $v(x) < \frac{1}{2}$ for some $x \in [-1,0]$ we have (here we denote by $l$ the linear function whose graph connects the point $(x,v(x))$ to $(0,1)$)
\begin{equation*}
    \|v\|_{W^{1,1}(\X)} \ge \|v'\|_{L^1((-x,0),\mathcal{L}^1)} \ge \|l'\|_{L^1((-x,0),\mathcal{L}^1)} = \frac{1-v(x)}{|x|} \ge 1-\frac{1}{2}=\frac{1}{2}.
\end{equation*}
From the two cases, we conclude
$\|T\|_{W^{1,1}} \ge (5\varepsilon)^{-1}$.
\end{example}

In the next example we go from the one-dimensional case of Example \ref{example:slice_example} to a two-dimensional one. The idea is to let the parameter $\varepsilon$ in Example \ref{example:slice_example} tend to zero. This is made possible by adding another direction in the space in order to modify the parameter $\varepsilon$.

\begin{example}\label{ex:BVnotW11}
Consider the space $\X = \mathbb R \times [0,1]$ with the Euclidean distance $\sfd_e$. Define the function $w$, where
\[
w(x,y) = \begin{cases}
    y, & \text{if }x \ge 0\\
    1, & \text{if }x < 0.
\end{cases}
\]
We consider the metric measure space $(\X,\sfd_e,\mu)$ with the measure $\mu:=w \mathcal L^2$. Then the set $E = [0,\infty)\times[0,1]$ is a $BV$-extension set, but not a $W^{1,1}$- nor a $W_w^{1,1}$-extension set.

Let $u \in BV(E)$ and consider the zero extension of the function $u$ to $\X$. In order to see that the zero extension gives a bounded operator, we observe
\begin{align*}
\|u\|_{BV(\X)} - \|u\|_{BV(E)} &= \int_{0}^1 u^+(0,y)w(0,y)\,\d y \\
& \le \int_{0}^1\int_0^\infty \left(u(x,y) +\frac{d}{dx}u(x,y)\right)w(x,y)\,\d x\, \d y = \|u\|_{BV(E)},
\end{align*}
where $u^+(0,y) = \lim_{x\searrow 0}u(x,y)$ exists for almost every $y \in [0,1]$ for the ACL-representative of $u$ (for the reference to absolute continuity on lines and this statement, we refer the reader to \cite[Theorem 6.1.13]{HKST15}).

Let us then show that $E$ is not a $W^{1,1}$-extension set in $(\X,\sfd_e,\mu)$. This example builds upon the computation in Example \ref{example:slice_example}.
We argue by contradiction and assume that there exists an extension operator $T \colon W^{1,1}(E) \to W^{1,1}(\X)$. Given $\delta >0$, we define the two cubes
$$ Q_\delta^-:=\left[ 0, \frac{\delta}{2} \right] \times \left[ \frac{\delta}{2}, \delta \right],\qquad Q_\delta^+:=\left[0, \frac{\delta}{2}+\frac{\delta}{4}\right]\times \left[\frac{\delta}{4}, \delta + \frac{\delta}{4} \right].$$
Let us define, given $A \subset \R \times [0,\infty)$ and $x \in \R \times [0,\infty)$, $\sfd_{\infty}(x,A):=\inf\{ \|x-y\|_{\infty}:\, y \in A\}$.
We consider the following function $u\colon E \to \R$, 
 \begin{equation*}
     u(x):=\begin{cases}
         1,\quad \text{if }x \in Q^-_\delta\\
         1-\frac{\delta} {2}\sfd_\infty(x,Q^-_\delta) \,\quad \text{if }x \in Q^+_\delta \setminus Q^-_\delta\\
         0,\quad \text{if }x \in E \setminus Q^+_\delta.
     \end{cases}
 \end{equation*}
 It can be readily checked that $u \in {\rm Lip}_{bs}(E)$, hence $u \in W^{1,1}(E)$. We define $v := T u \in W^{1,1}(\X)$.

For $\mathcal{L}^1$-a.e.\ $t \in [\delta/2,\delta]$, $v(\cdot,t) \in W^{1,1}(\R,\sfd_e,\mathcal{L}^1\restr{(-\infty,0)}+t \mathcal{L}^1\restr{[0,\infty)})$.
By repeating the argument as in the previous example, we get
\begin{equation*}
    \|\partial_x v(\cdot,t)\|_{W^{1,1}(\R,\sfd_e,\mathcal{L}^1\restr{(-\infty,0)} +t \mathcal{L}^1\restr{[0,\infty)})} \ge \frac{1}{2}. 
\end{equation*}
By integrating over the $t$ variable, we have 
 \begin{equation*}
    \|\nabla v\|_{L^1(\mu)} \gtrsim \delta. 
\end{equation*}
On the other hand, $\|u\|_{W^{1,1}(E)} \lesssim \delta^3$. Indeed, this follows by the two estimates
\begin{equation*}
    \begin{aligned}
        &\int_E |u|\,y\,\d \mathcal{L}^2 \lesssim \delta \mathcal{L}^2(Q^+_\delta) \lesssim \delta^3\text{ and}\\
        & \int_E |\nabla u|\,y\,\d \mathcal{L}^2 \lesssim \delta \mathcal{L}^2(Q^+_\delta \setminus Q^-_\delta) \lesssim \delta^3.\\
    \end{aligned}
\end{equation*}

This in particular gives $\|T\| \gtrsim \delta^{-2}$ for every $\delta >0$, contradicting its finiteness.
The argument above also shows that $E$ is not a $W_w^{1,1}$-extension set in $(\X,\sfd_e,\mu)$. 
\end{example}

%

\subsection{On the relations between $W^{1,1}$- and $W^{1,1}_w$-extension sets}

We end this section by discussing the relation between \(( W^{1,1} )\) and \(( W^{1,1}_w )\) extension properties for closed sets.
In \cite[Prop. 4.2]{CKR23} it was shown that $(W^{1,1}) \Rightarrow (W_w^{1,1})$ holds for open subsets. The proof relied on the smoothening of a $W_w^{1,1}$-function to a $W^{1,1}$-function via Proposition \ref{prop:smoothing}. This argument does not work for closed sets as sketched in the following example.

\begin{remark}[On the relation between \(( W^{1,1} )\) and \(( W^{1,1}_w )\) via smoothing argument]
As our space we consider \((\R, \sfd_{e}, \mm)\), with \(\mm=\mathcal L^1+\delta_0\). Then one can check that \(E=[0,1]\) is a closed \(W^{1,1}_w\)- and \(W^{1,1}\)-extension set. 
However, we cannot deduce \(W^{1,1}\)-extension property from \(W^{1,1}_w\)-extension property via a smoothing argument presented in the proof of Proposition \ref{prop:general_implications} (or Proposition \ref{prop:11_11w_relation} below). Indeed, take a 
function \(u\in W^{1,1}([0,1])\). The zero extension on the complement of \([0,1]\) gives an \(W^{1,1}_w\)-extension operator, call it \(\bar F\). Since the smoothing operator \(T_\varepsilon\) leaves the function (being equal to zero) in the complement unchanged, the function defined as
\(Fu\coloneqq \chi_{\R\setminus [0,1]}T_\varepsilon(\bar Fu\restr{\R\setminus [0,1]})+\chi_{[0,1]}u\) does not belong to \(W^{1,1}(\X
)\).
\end{remark}

In \cite[Ex. 4.7]{CKR23} an example was given of a domain that is a $W_w^{1,1}$-extension domain, but not a $W^{1,1}$-extension domain. For closed sets it is still unknown if such an example exists.
While the relation between the two in general remains an open problem, formulated in Question \ref{q:11_11w} below, under the additional assumption \(\mm(\partial E)=0\), we have the equivalence of the two properties. Before proving the latter fact, we observe that, even in this case,  we cannot deduce \(( W^{1,1} )\)-extension property from \(( W^{1,1}_w )\)-extension property via passing to \((s\text{-}BV) \)-extension property:

\begin{example}[\(BV\)-, \(W^{1,1}\)- and \(W^{1,1}_w\)-extension set with \(\mm(\partial E)=0\) but not \(s\text{-}BV\)-extension set]\label{ex:tripod}
Let us consider as our space \(\X\) to be a tripod with the unit length legs \(\X_i\), \(i=1,2,3\), each equipped with the Euclidean distance and the \(1\)-Hausdorff measure. As the set \(E\) take the union of the two legs, i.e.\ \(E=\X_1\cup\X_2\). Then, one can check that \(E\) is a \(W^{1,1}\)-, \(W^{1,1}_w\)- (and thus also \(BV\)-) extension set, but it is not an \(s\text{-}BV\)-extension set. Indeed, consider as a function \(u=\chi_{\X_1}\in BV(E)\). Then, any extension \(Fu\) of \(u\) to the complement of \(E\) in \(\X\) will have positive \(|DFu|\)-measure at the common point of the three legs, in other words \(|DFu|(\partial E)>0\).
\end{example}

Nevertheless, we have the following result.

\begin{proposition}[\(( W^{1,1}_w ) \Longleftrightarrow ( W^{1,1} )\) under \(\mm(\partial E)=0\)] \label{prop:11_11w_relation}
Let \((\X,\sfd, \mm)\) be a metric measure space and let \(E\subset \X\) be a closed nonempty set with \(\mm(\partial E)=0\). Then \(E\) is \(W^{1,1}_w\)-extension set if and only if it is \(W^{1,1}\)-extension set.
\end{proposition}
\begin{proof}\ \\
\textsc{Proof of \(( W^{1,1}_w ) \Longrightarrow ( W^{1,1} )\).} Assume that \(E\) is a \(W^{1,1}_w\)-extension set and denote by \(\bar F\) the associated extension operator. Let \(u\in W^{1,1}(E)\subset W^{1,1}_w(E)\). We then proceed as in the first part of the proof of Proposition \ref{prop:general_implications}. Given \(\varepsilon>0\) we define 
\(u_1\coloneqq T_\varepsilon(\bar F u\restr{\Omega})\) with \(\Omega\coloneqq 
\X
\setminus E\) and \(u_2\coloneqq u\). We set 
\(Fu\coloneqq \chi_{\X\setminus E} u_1+\chi_E u_2\), which turns out to be a desired \(W^{1,1}\)-extension operator. The proof follows the same steps as in Proposition \ref{prop:general_implications}, wherein the possibility to apply Lemma \ref{lemma:fromlocal_to_global_W11} is given by the fact that \(|D\bar Fu|\ll \mm\) and thus \(|D\bar Fu|(\partial E)=0\), due to the hypothesis.\\

\noindent
{\sc Proof of \(( W^{1,1} ) \Longrightarrow ( W^{1,1}_w )\).} 
Assume that \(E\) is a \(W^{1,1}\)-extension set and denote by \(\bar F\) the associated extension operator. Let \(u\in W^{1,1}_w(E)\subset BV(E)\).
Fix \(\varepsilon>0\) and let \(T_\varepsilon\colon BV({\rm int}\, E)\to \Lip_{\rm loc}({\rm int}\,E)\) be the smoothing operator from Proposition \ref{prop:smoothing}. Define \(\bar u\coloneqq T_\varepsilon (u|_{{\rm int\,}E})\in \Lip_{\rm loc}({\rm int}\,E)\subset W^{1,1}({\rm int}\, E)\).
Let us consider the function \(\tilde u\in BV((\X, \sfd, \mm\restr{E}))\) defined as the zero extension of $\bar u$ from ${\rm int}\,E$ to $\X$. 
Given that \(|D\tilde u|(\partial E)=|D\bar u|(\partial E)=0\), since \(|D\bar u|\ll \mm\), we can apply Lemma \ref{lemma:fromlocal_to_global_W11} and deduce that 
\(\tilde u\in W^{1,1}((\X,\sfd, \mm|_E))=W^{1,1}(E)\) and it holds that 
\[
|D\tilde u|_{1,E}\leq \chi_{{\rm int}\,E}|D\bar u|_{1,{\rm int}\, E}\quad \mm\text{-a.e.\ in }  E.
\]
Then, we define 
\[
Fu\coloneqq \chi_{\X\setminus {\rm int}\,E}\bar F\tilde u + \chi_{{\rm int}\, E}\, u\in BV(\X),
\]
where \(u\) is understood to be extended to \(0\) in \(\X\setminus E\).
Notice that \(Fu=u\) holds \(\mm\)-a.e.\ on \(E\) (recall \(\mm(\partial E)=0\)). We claim that \(Fu\in W^{1,1}_w(\X)\). Indeed, 
we can write \(Fu=\chi_{{\rm int}\,E}(u\restr{{\rm int}\, E}-\bar u)+\bar F\tilde u\). By the definition of \(\bar u\) and Proposition \ref{prop:smoothing}, we have that \(u\restr{{{\rm int}\,E}}-\bar u\in W^{1,1}_w(\X)\) (when both \(u|_{{\rm int}\, E}\) and \(\bar u\) are understood as extended to \(0\) in \(\X\setminus{\rm int}\,E\)).
Since also \(|D\bar F\tilde u|\ll \mm\), we   deduce that \(Fu\in W^{1,1}_w(\X)\). It remains to verify the norm bounds. First, we have that 
\[
\|Fu\|_{L^{1}(\mm)}\leq \|\bar F\tilde u\|_{L^1(\mm)}+\|u\|_{L^1(E)} \leq C_{\bar F}\|\tilde u\|_{W^{1,1}(E)} +\|u\|_{W^{1,1}(E)}.
\]
Then, using the norm bounds for the smoothing operator in Proposition \ref{prop:smoothing} we get 
\[
\|\tilde u\|_{W^{1,1}(E)}=\|T_\varepsilon(u|_{{\rm int}\,E})\|_{W^{1,1}({\rm int}\,E)}\leq \|u\|_{L^1(E)}+\varepsilon+C\big(|Du|_E(E)+\varepsilon\big)\leq (C+1) \|u\|_{BV(E)}+(C+1)\varepsilon.
\]
Together with the above, this gives
\[
\|Fu\|_{L^1(\mm)} \leq C_{\bar F}(C+2) \|u\|_{BV(E)}+C_{\bar F}(C+1)\varepsilon.
\]
Moreover, taking into account \eqref{eq:from_B_to_X}, we have 
\[
\begin{split}
|DFu|(\X)\leq &|D\bar F\tilde u|(\X)+|Du|({\rm int}\,E)\leq C_{\bar F}\|\tilde u\|_{W^{1,1}(E)}+\|u\|_{BV(E)}\\
\leq  &\big(C_{\bar F}(C+2)+1\big) \|u\|_{BV(E)}+ C_{\bar F}(C+2)\varepsilon.
\end{split}
\]
Choosing \(\varepsilon\) to be precisely \(\|u\|_{BV(E)}\) we conclude the proof.
\end{proof}

\begin{question}\label{q:11_11w}
Let $E \subset X$ be a nonempty closed set with \(\mm(\partial E)>0\). Does $(W^{1,1}) \Rightarrow (W_w^{1,1})$ hold for $E$? Does $(W_w^{1,1}) \Rightarrow (W^{1,1})$ hold for $E$?
\end{question}

\section{Extension sets in PI spaces}\label{sec:PI}

In this section we move to the setting of PI spaces where the extension sets have more properties than in general metric measure spaces. We will also give examples showing that the results fail in general metric measure spaces.

\subsection{Measure density}
Let us start with the results for PI spaces regarding the measure density of general extension sets.
In PI spaces, a measure density result was proven for $W^{1,p}$-extension domains in \cite{HKT2008b} under the additional assumption that the space is Ahlfors regular. 
A measure density was also studied for BV-extensions in \cite{KoskelaMirandaShanmugalingam2010,GarciaBravoRajala2024}, respectively for the homogeneous $BV$ space in the planar case and for the full norm in the general Euclidean case.

In Proposition \ref{prop:measdens} we provide a measure density result in PI setting 
where we remove the extra assumption of Ahlfors regularity of the measure with the help of the following direct consequence of the doubling property, proven for example in \cite[Lemma 3.7]{KorteLahti2014}.
%

\begin{lemma}\label{lma:boundaryzero}
Let $(\X,\sfd)$ be a length space and $\mm$ a doubling measure on $\X$. Then, for every $x\in \X$ and $r>0$, $\mm(S(x,r)) = 0$ where $S(x,r) = \{y \in\X\,:\, \sfd(x,y) = r\} \subset \X$.
\end{lemma}

\begin{proposition}\label{prop:measdens}

Let \((\X, \sfd, \mm)\) be a PI space and $E \subset \X$ a bounded $BV$-extension set with $\mm(E)>0$. Then there exists a constant $C > 0$ so that for all $x \in \overline{X \setminus E^0}$ and $0 < r < {\rm diam}(E)$ we have
\[
\mm(B(x,r)\cap E) \ge C \mm(B(x,r)).
\]
\end{proposition}
\begin{proof}
{\color{blue}{\sc Step 1:}}
We first prove the theorem for $x \in \X \setminus E^0$.
The proof follows closely the proofs of \cite[Theorem 2]{HKT2008b} and \cite[Proposition 2.3]{BR21}. Since a few minor modifications are needed, we repeat most of the proof for readability.
Since we are working in a PI space, we can assume without loss of generality that the distance \(\sfd\) on the space is geodesic, see \cite[Corollary 8.3.16]{HKST15}.
Now fix any \(x\in \X \setminus E^0\) and \(r\in (0,{\rm diam}(E))\). Let us also define \(r_0\coloneqq r\).
We point out that the constant \(C\) appearing throughout the proof is universal, but may change from line to line.
By Lemma \ref{lma:boundaryzero} we have $\mm(\partial B)=0$ for every ball $B \subset \X$, which implies that
\begin{equation*}
    (0,\infty) \ni r \mapsto \mm(E \cap B(x,r)) \in \mathbb{R}\quad\text{is continuous.}
\end{equation*}
Thus, by induction, we are able to define for every $i \in \mathbb N$ the radius $r_i \in (0,r_{i-1})$ by the equality
\[
\mm(E\cap B(x,r_i)) = \frac12\mm(E\cap B(x,r_{i-1})) = 2^{-i}\mm(E\cap B(x,r_0)).
\]
Since $x \in  X \setminus E^0$, we have that $r_i \searrow 0$ as $i\to \infty$.

For each $i \in\mathbb N$, consider the function $u_i\colon E \to \mathbb{R}$
\[
 u_i(y) \coloneqq \begin{cases}1, & \text{for }y \in B(x,r_i) \cap E,\\
 \frac{r_{i-1}-\sfd(x,y)}{r_{i-1}-r_i}, & \text{for }y \in (B(x,r_{i-1})\setminus B(x,r_i)) \cap E,\\
 0, & \text{otherwise}.
 \end{cases}
\]
Since $u_i$ is the restriction of a globally defined Lipschitz function and $\mm(E)<\infty$ is $E$ is boundedly finite, by the definition of a $BV$ space, we have that $u_i \in BV(E)$. Moreover
\begin{equation}
\label{eq:bound_BV_norm_bump_function}
\begin{aligned}
\Vert u_i\Vert_{BV(E)} &=\Vert u_i\Vert_{L^1(E)} +|D  u_i|_E(E)  \le \int_{E} |u_i|\,\d \mm+  \int_{E} \lip_a(u_i)\,\d \mm \\
& \le  \mm(B(x,r_{i-1})\cap E) + |r_i - r_{i-1}|^{-1} \mm((B(x,r_{i-1})\setminus B(x,r_i))\cap E) \\
&\leq C|r_i - r_{i-1}|^{-1} \mm(E\cap B(x,r_i)).
\end{aligned}
\end{equation}
Let us denote by $T\colon BV(E)\to BV(\X)$ the extension operator.

By Proposition \ref{prop:Sobo}, there are constants \(C>0\) and \(\lambda\geq 1\) so that
\begin{equation}\label{eq:superlevelsetest}
\frac{\mm(\{z \in B(x,r_{i-2})\,:\,|T u_i(z) - (T u_i)_{B(x,r_{i-2})}|>\frac12\})(\frac12)^{\frac{s}{s-1}}}{\mm(B(x,r_{i-2}))} \le C\left(r_{i-2} \frac{|D (T u_i)|(B(x,5\lambda r_{i-2}))}{\mm(B(x,5\lambda r_{i-2}))} \right)^{\frac{s}{s-1}},
\end{equation}
with \(s\) being as in \eqref{eq:s_doubling}.

We claim that we have
\begin{equation}
    \label{eq:percentage_ball_bounded_by_oscillation}
\mm(E\cap B(x,r_i)) \le C
\mm\left(\left\{z \in B(x,r_{i-2})\,:\,|T u_i(z) - (T u_i)_{B(x,r_{i-2})}|>\frac12\right\}\right).
\end{equation}
Indeed, by the definition of $r_i$, there exists a constant $C>0$ such that 
\begin{equation}
\label{eq:comparing_two_sets}
    \mm(E\cap B(x,r_i))\le C \mm(E \cap (B(x,r_{i-2})\setminus B(x,r_{i-1}))).
\end{equation}
We follow the argument in \cite[Theorem 2]{HKT2008b}. Since $u_i=1$ on $B(x,r_i) \cap E$ and $u_i=0$ on $(B(x,r_{i-2})\setminus B(x,r_{i-1})) \cap E$, we have $|T u_i(z)-(T u_i)_{B(x,r_{i-2})}| >\frac{1}{2}$ for every point in either one of these two sets. This implies
\begin{equation}
\label{eq:dichotomy_oscillation}
\begin{aligned}
    \min\{\mm(E\cap B(x,r_i)),\, &\mm(E \cap (B(x,r_{i-2})\setminus B(x,r_{i-1})))\} \\
    &\le
\mm\left(\left\{z \in B(x,r_{i-2})\,:\,|T u_i(z) - (T u_i)_{B(x,r_{i-2})}|>\frac12\right\}\right).\\
\end{aligned}
\end{equation}
Combining \eqref{eq:comparing_two_sets} with \eqref{eq:dichotomy_oscillation}, we get \eqref{eq:percentage_ball_bounded_by_oscillation}.

Since $T$ is a $BV$-extension set, we have
\begin{equation}
\label{eq:upper_bound_asymptotic_Lipschitz constant}
\frac{|D (T u_i)|(B(x,5\lambda r_{i-2}))}{\mm(B(x,5\lambda r_{i-2}))} \le \frac{|D (T u_i)|(\X)}{\mm(B(x,5\lambda r_{i-2}))} \le \frac{\|T\| \Vert u_i\Vert_{BV(E)}}{\mm(B(x,5\lambda r_{i-2}))}.
\end{equation}
From \eqref{eq:s_doubling} we obtain
\begin{equation}
\label{eq:growth_volume_small_balls_doubling}
\frac{r_{i-2}}{\left(\mm(B(x,r_{i-2}))\right)^{\frac{1}{s}}} \le C\frac{r}{\left(\mm(B(x,r))\right)^{\frac{1}{s}}}.
\end{equation}
Combining the above inequalities we have
\begin{align*}
\left(\mm(E\cap B(x,r_i))\right)^{1-\frac{1}{s}}
& \stackrel{\eqref{eq:superlevelsetest},\eqref{eq:percentage_ball_bounded_by_oscillation}}{\le} C \left(\mm(B(x,r_{i-2}))\right)^{1-\frac{1}{s}}r_{i-2} \frac{|D (T u_i)|(B(x,5\lambda r_{i-2}))}{\mm(B(x,5\lambda r_{i-2}))}\\
& \stackrel{\eqref{eq:upper_bound_asymptotic_Lipschitz constant}}{\le} C \frac{r_{i-2}}{\left(\mm(B(x,r_{i-2}))\right)^{\frac{1}{s}}} 2\|T\| \Vert u_i\Vert_{BV(E)}\\
& \stackrel{\eqref{eq:growth_volume_small_balls_doubling}}{\le} C\frac{r}{\left(\mm(B(x,r))\right)^{\frac{1}{s}}} 2\|T\| \Vert u_i\Vert_{BV(E)}\\
& \stackrel{\eqref{eq:bound_BV_norm_bump_function}}{\le} C\|T\| \frac{r}{\left(\mm(B(x,r))\right)^{\frac{1}{s}}} |r_i - r_{i-1}|^{-1} \mm(E\cap B(x,r_i)).
\end{align*}
Arranging the terms in the inequality above, we have
\[
 \left(\mm(B(x,r))\right)^{\frac{1}{s}}
 \frac{|r_i - r_{i-1}|}r \le C\|T\| (\mm(E\cap B(x,r_i)))^\frac{1}{s} = C\|T\|2^{-i/s} (\mm(E\cap B(x,r_0)))^\frac{1}{s}.
\]

By summing up all these quantities we conclude that
\begin{align*}
  \left(\mm(B(x,r))\right)^{\frac{1}{s}} 
  & = \left(\mm(B(x,r))\right)^{\frac{1}{s}} \frac{r}{r} = \sum_{i=1}^\infty
   \left(\mm(B(x,r))\right)^{\frac{1}{s}}
 \frac{|r_i - r_{i-1}|}r \\
 & \le C\|T\|\sum_{i=1}^\infty 2^{-i/s}\mm(E\cap B(x,r_0))^\frac{1}{s} = C\|T\|\mm(E\cap B(x,r_0))^\frac{1}{s}.
\end{align*}
This gives the claimed inequality.

{\color{blue}{\sc Step 2:}} We now prove that the result holds for every point in $\overline{\X \setminus E^0}$. To this end  consider 
$x \in \overline{\X \setminus E^0} $ and let $ x_k \in \overline{\X \setminus E^0}$ be such that $x_k$ converges to $x$. By Step 1 we have that 
\begin{equation*}
     \mm(B(x_k,r)\cap E) \ge C \mm(B(x_k,r))\qquad \text{for every }0 < r < {\rm diam}(E)\text{ and }k \in \mathbb{N}.
\end{equation*}
Since $\lims_{k \to \infty} |\mm(B(x_k,r) \cap E) - \mm(B(x,r) \cap E)| \le \mm(\partial B(x,r))$ and $\lims_{k \to \infty} |\mm(B(x_k,r)) - \mm(B(x,r))| \le \mm(\partial B(x,r))$, by letting $k$ go to $\infty$ and applying Lemma \ref{lma:boundaryzero} we conclude. 
\end{proof}

The next example shows that the measure density fails in general metric measure spaces.

\begin{example}\label{ex:sierpinski}
Proposition \ref{prop:measdens} does not hold if we remove the assumption that $(\X,\sfd,\mm)$ is a PI space.
Let $\X = \mathbb R$ with the Euclidean distance $\sfd$, and $\mm = \rho\mathcal L^1$ with
\[
\rho(x) = \begin{cases}
1, & \text{if }x \notin [0,1]\\
\min(x,1-x), & \text{if }x \in [0,1].
\end{cases}
\]
Then $[0,1]$ is a $BV$-extension set in $(\X,\sfd,\mm)$. Moreover $\R \setminus E^0=(0,1)$, so $0 \in \overline{\R \setminus E^0}$ but
\[
\lim_{r \to 0}\frac{\mm(B(0,r)\cap [0,1])}{\mm(B(0,r))} = 0.
\]

\end{example}

\begin{remark}[Measure density and measure of the boundary]
For open $BV$-extension sets $E \subset \X$ in PI spaces, Proposition \ref{prop:measdens} implies that $\mm(\partial E) = 0$. For general $BV$-extension sets $E$ in PI spaces the conclusion $\mm(\partial E) = 0$ need not hold. 

A trivial example of this is to consider $E = \mathbb D \setminus \mathbb Q^2 \subset \mathbb R^2$, with \(\R^2\) being endowed with the Euclidean distance and the Lebesgue measure. However, in this case $E$ has a closed representative with $\mm(\partial E) = 0$ and this closed representative is also a $BV$-extension set (by Proposition \ref{prop_closedrepr}). A less trivial and more satisfying example of an extension set with positive boundary is given by a fat Sierpi\'nski carpet in the plane, as proved in \cite{GarciaBravoRajala2024}.
\end{remark}

\subsection{From homogeneous norm to full norm}

 In the case of PI spaces we have that bounded $\ca{BV}$-extension sets are $BV$-extension sets. For a similar proof in the Euclidean setting for Sobolev functions, see the proof of \cite[Theorem 4.4]{HerronKoskela}.

 \begin{proposition}\label{prop:homBVtofullBV}
 Let $(\X,\sfd,\mm)$ be a PI space and $E$ a bounded $\ca{BV}$-extension set. Then $E$ is a $BV$-extension set.
 \end{proposition}
 \begin{proof}
 Let $T \colon \ca{BV}(E) \to \ca{BV}(\X)$ be the extension operator. 
 Since $E$ is bounded, there exists a ball $B(x,r)$ such that $E\subset B(x,r)$. By \cite{Rajala2021}, there exists a uniform domain $\Omega \subset B(x,r+1)$ with $E \subset \Omega$. By \cite{Panu} (see also \cite{JanaNages}), $\Omega$ is a $BV$-extension set with respect to the full norm. Thus, we only need to check that $T \colon BV(E) \to BV(\Omega)$ is also an extension operator with respect to the full norm.

Let $u \in BV(E)$.
Then, by the $(1,1)$-Poincar\'e inequality \eqref{eq:poincare_BV_rem4.1} computed on the ball $B(x,r+1)$ and the fact that $T u= u$ on $E$, we have 

\begin{equation}
\label{eq:poincare_with_average_on_E}
    \int_{B(x,r+1)}|Tu-u_E|\,\d\mm \le \frac{2\mm(B(x,r+1))}{\mm(E)}C(r+1) |D(Tu)|(B(x,\lambda(r+1))),
\end{equation}
where $\lambda$ is the constant appearing in the definition of a PI space.
Therefore,
\begin{align*}
 \int_\Omega|Tu|\,\d\mm &\le  \int_{B(x,r+1)}|Tu|\,\d\mm
 \le \int_{B(x,r+1)}|Tu-u_E|\,\d\mm + \int_{B(x,r+1)}|u_E|\,\d\mm\\
 & \stackrel{\eqref{eq:poincare_with_average_on_E}}{\le} C(r+1) |D(Tu)|(B(x,\lambda(r+1))) + {\mm(B(x,r+1))}|u_E|.
\end{align*}
Hence,
\begin{align*}
\|Tu\|_{BV(\Omega)} & \le \int_{\Omega} |Tu|\,\d\mm + |D(T u)|(B(x,r+1)) \le (C(r+1)+1)|D(Tu)|(\X)+ {\mm(B(x,r+1))}|u_E|\\
& \le (C(r+1)+1)\|T\||Du|_E(E) + \frac{\mm(B(x,r+1))}{\mm(E)}\int_E|u|\,\d\mm \le \tilde{C} \|u\|_{BV(E)},
\end{align*}
proving the claim. 
\end{proof}

 The converse to Proposition \ref{prop:homBVtofullBV} does not hold even in the Euclidean setting. For instance, consider in $(\mathbb{R}^2,\sfd_e,\mathcal{L}^2)$ the set $E = B((0,0),1) \cup B((3,0),1)$. Then $E$ is a $BV$-extension set, but not a $\ca{BV}$-extension set.

 Moreover, in general metric measure spaces, the conclusion of Proposition \ref{prop:homBVtofullBV} need not hold.

\begin{example}[Weighted Hawaiian earring]
\label{example:earring}
 Let us give an example of a compact $\ca{BV}$-extension set that is not a $BV$-extension set.
The space is constructed by gluing together infinitely many scaled copies of $\mathbb S^1$.

Let us denote the copies by $S_i = 2^{-i}\mathbb S^1$ and select a point $x_i \in S_i$ for all $i \in \mathbb N$.
Let us denote the copies by $S_i = 2^{-i}\mathbb S^1$ obtained as the quotient $\faktor{\mathbb R}{2^{-i+1}\pi\mathbb Z}$ and write $x_i = [0]_i$ for all $i \in \mathbb N$, where $[\cdot]_i$ denotes the equivalence class under the quotient.
The metric space $(\X,\sfd)$ is then the wedge sum of $S_i$ where all the points $x_i$ are identified with each other. The distance $\sfd$ is the geodesic distance obtained from the geodesic distances $\sfd_{S_i}$ on $S_i$, namely
\[
\sfd(x,y) = \begin{cases}
\sfd_{S_i}(x,y), & \text{if }x,y \in S_i\\
\sfd_{S_i}(x,x_i) + \sfd_{S_j}(y,x_j), & \text{if }x \in S_i, y \in S_j, i \ne j.
\end{cases}
\]
On each $S_i$ we define a weight
\[
w_i(x) = \begin{cases}
2^{-i}, & \text{if }\sfd(x,x_i) < 2^{-2i} \text{ or }
\sfd(x,x_i) > 2^{-i}\pi - 2^{-2i}\\
i, & \text{otherwise.}
\end{cases}
\]
The reference measure of our space is then given as 
\[
\mm = \sum_{i=1}^\infty w_i\mathcal H^1\restr{S_i}.
\]

Next we define the compact $\ca{BV}$-extension set $E \subset \X$ as
\[
E = \bigcup_{i=1}^\infty\{x\in S_i \,:\, \sfd(x,x_i) \le 2^{-2i+1}\}.
\]
The bounded extension operator $T \colon \ca{BV}(E) \to \ca{BV}(\X)$ is given in each $S_i$ by a reflection of the function to a neighbourhood of the antipodal point and by extending the function as a constant on the remaining parts. Let us be more precise. Since $u\in \ca{BV}(E)$, then $u\restr{S_i}\in \ca{BV}(E \cap S_i,\sfd_{S_i},\mm\restr{E \cap S_i})$ for every $i\in \mathbb{N}$. 
Denote $u \restr{S_i}\colon E \cap S_i \to \R$ as $u_i$. 
Since $u_i$ has a continuous representative in $E \cap S_i$, we can define 
\[
u_i^+ = \lim_{x \nearrow 2^{-2i+1}}u_i([x]_i) \qquad \textrm{and} \qquad u_i^- = \lim_{x \searrow -2^{-2i+1}}u_i([x]_i).
\]
We define the extension on $S_i$ as
\begin{equation*}
    T u([x]_i):=\begin{cases}
        u_i([x]_i) & \text{if }x \in (-2^{-2i+1},2^{-2i+1}),\\
        u_i^+ & \text{if }x \in [2^{-2i+1},2^{-i}\pi -2^{-2i+1}],\\
        u_i([2^{-i}\pi -x]_i)& \text{if }x \in (2^{-i}\pi -2^{-2i+1},2^{-i}\pi +2^{-2i+1}),\\
        u_i^-& \text{if }x \in [2^{-i}\pi +2^{-2i+1}, -2^{-2i+1}].
    \end{cases}
\end{equation*}

In order to see that $E$ is not a $BV$-extension set, take $i \in \mathbb N$ and consider the function $u_i \in BV(E)$ that is $0$ everywhere else except for $E \cap S_i$ where we define it to be
\[
u_i(x) = \begin{cases}
 2^{2i}\sfd(x,x_i), & \text{if }\sfd(x,x_i) < 2^{2i},\\
 1, & \text{otherwise.}
\end{cases}
\]
Then 
\[
\|u_i\|_{BV(E)} \le 2^{-i+1} + i2^{-2i+2} \le 2^{-i+2},
\]
while for any extension $\tilde u_i$ of $u_i$ to $BV(\X)$ we can estimate as follow. We denote
\[
t = \textrm{ess\,inf}\left\{\tilde u_i(x) \,:\, x \in S_i, 2^{-2i} \le \sfd(x,x_i) \le 2^{-i}\pi - 2^{-2i}\right\}.
\]
Then
\[
\|\tilde u_i\|_{L^1(\mm)} \ge |t| i (2^{-i+1}\pi -2^{-2i+2}) \ge |t|i2^{-i}
\]
and
\[
|D\tilde u_i|(S_i) \ge i|1-t| \ge |1-t|i2^{-i}.
\]
Together these imply 
\[
\|\tilde u_i\|_{BV(\X)} \ge i 2^{-i} \ge \frac{i}{4}\|u_i\|_{BV(E)}.
\]
Thus, any extension operator from $BV(E)$ to $BV(\X)$ has norm at least $i/4$. Since this holds for every $i \in \mathbb N$, the set $E$ is not a $BV$-extension set.
\end{example}
 
\subsection{Connectivity and decomposition properties of $BV$- and $\ca{BV}$-extension sets}\label{sec:connectivity}

Proposition \ref{prop_closedrepr} allows us to start from a closed $BV$-extension set $E$ and move to a new closed $BV$-extension set $\overline{E^1}$, that is the closure of the points of density $1$. The original set $E$ might not have nice metric/topological properties, but the set $\overline{E^1}$ does.

Since in this section we assume the metric measure space $(\X,\sfd,\mm)$ to be a PI space, the measure $\mm$ is doubling. We prove the following simple lemma.

\begin{lemma}
Let $(\X,\sfd,\mm)$ be a doubling metric measure space. Let $E \subset \X$ be a closed set. Then any of the properties 
($\ca{BV}$), ($BV$), ($L^{1,1}$), ($W^{1,1}$), ($L_w^{1,1}$), ($W_w^{1,1}$), (s-$\ca{BV}$), (s-$BV$)  
for $E$ implies the same property for $\overline{E^1}$.
\end{lemma}

\begin{proof}
    Since the space is doubling, by definition of $E^1$ and the application of Lebesgue differentiation theorem, we have that $\mm(E\setminus E^1)=0$. 
    Moreover, since $E$ is closed we have that $E \setminus \overline{E^1}\subset E \setminus E^1$, thus $\mm(E \setminus \overline{E^1})=0$. By Proposition \ref{prop_closedrepr}, we have that $\overline{E^1}$ that satisfies the conclusion of lemma.   
\end{proof}

Let us recall the notion of indecomposable sets for finite perimeter.

\begin{definition}
\label{def:decomposability}
Let $(\X,\sfd,\mm)$ be a metric measure
space and $E \subset \X$ a set of finite perimeter. Given any Borel set $B \subset \X$, we say that $E$ is
decomposable with respect to $B$ provided there exists a partition of $E \cap B$ into sets 
$F,G \subset E$
such that $\mm(F), \mm(G) > 0$ and $\Pe_B(E, B) = \Pe_B(F, B) + \Pe_B(G, B)$. We say that $E$ is
indecomposable with respect to $B$ if it is not decomposable with respect to $B$.
We simply say that $E$ is (in)decomposable if $E$ is (in)decomposable with respect to $X$.
\end{definition}

It is known \cite{Lahti_decomp,BPR2020} that on PI spaces there exists a unique (up to measure zero sets) decomposition of sets of finite perimeter into indecomposable sets.

\begin{theorem}
Let $(\X,\sfd,\mm)$  be a PI space. Let $E \subset \X$ be a set of finite perimeter. Then there exists a unique (finite or countable) partition $\{E_i\}_{i\in I}$ of $E$ into
indecomposable subsets of
$\X$ such that $\mm(E_i) > 0$ for every $i \in I$ and $\Pe(E, \X) =
\sum_{i \in I}\Pe(E_i,\X)$, where uniqueness is in the $\mm$-a.e. sense. Moreover, the sets $\{E_i\}_{i\in I}$
are maximal indecomposable sets, meaning that for any Borel set $F \subset E$ with
$P(F, \X) < \infty$ that is indecomposable there is a (unique) $i \in I$ such that $\mm(F \setminus E_i) =
0$.
\end{theorem}



\begin{remark}
    It follows by the definition of indecomposable sets that given a metric measure space $(\X,\sfd,\mm)$ such that ${\rm supp}\,\mm=\X$, if $\X$ is indecomposable, then it is connected.
    If ${\rm supp}\,\mm$ is not the full space, then indecomposability of $\X$ does not imply connectedness. Indeed, consider in $\mathbb{R}^2$ for the set $\X = B((-2,0),1) \cup B((2,0),1)$ the metric measure space $(\X,\sfd_e,\mathcal{L}^2\restr{B((-2,0),1)})$. It is straightforward to check that $X$ is indecomposable, but not connected.
\end{remark}


\begin{remark}
    If $(\X,\sfd,\mm)$ is a PI space, then $\X$ is indecomposable. Indeed, if not, there exists a Borel set $A \subset \X$ with $\mm(A)>0$, $\mm(\X\setminus A)>0$ such that $ {\rm P}(A,\X)=0$. Since $(\X,\sfd,\mm)$ satisfies a $1$-Poincar\'{e} inequality, we have that $\chi_A$ is constant, thus $A=\emptyset$ or $A=\X$, up to a negligible set. This contradicts respectively $\mm(A)>0$ or $\mm(\X\setminus A)>0$.
\end{remark}

\begin{lemma}\label{lma:connected_hom}
 Suppose that $(\X,\sfd,\mm)$ is a metric measure space so that $\X$ is indecomposable. Let $E \subset \X$ be a closed $\ca{BV}$-extension set. Then $E$ is indecomposable with respect to $E$.
\end{lemma}

\begin{proof}
 Since $E$ is closed, ${\rm P}_E(E,E)=0$. Suppose that $E$ is decomposable with respect to $E$. Then there exists a set $A \subset E$ so that 
 $\mm(A)>0$, $\mm(E \setminus A)>0$ and ${\rm P}_E(A,E) = 0$.
 Since $E$ is a $\ca{BV}$-extension set, by \cite[Remark 3.2]{CKR23} the set $E$ has the extension property for sets of finite perimeter. Let us call $\tilde{A}$ the extension of $A$ from $E$ to $\X$.
 Since $\mm(A)>0$, $\mm(E \setminus A)>0$, we have that $\mm(\tilde{A})\ge \mm(A)>0$ and $\mm(\X\setminus \tilde{A}) \ge \mm(E \setminus A)>0$, thus the assumption on indecomposability of $\X$ gives
 \[
 {\rm P}(\tilde{A},\X)>0.
 \]
 However, since $E$ has the extension property for sets of finite perimeter,
 \[{\rm P}(\tilde{A},\X)\le C {\rm P}_{E}(A,E)=0.\]
 This is a contradiction.
\end{proof}
\begin{remark}
    We point out that, as a consequence of Proposition \ref{prop_closedrepr} and Lemma \ref{lma:connected_hom} (since for a given closed set $E$ we have that $\mm(E \triangle \overline{E^1})=0$), under the assumption that $E$ is closed $\ca{BV}$-extension set, also $\overline{E^1}$ is indecomposable with respect to $\overline{E^1}$.
\end{remark}

If we consider extension sets with the full norm, they need not be connected even in the Euclidean setting. To see this, we can consider two disjoint disks in $\mathbb R^2$. However, under a PI-assumption we have that any closed bounded extension set $E$ has a finite decomposition into indecomposable sets with respect to $E$. This is an immediate consequence of the following lemma.


\begin{lemma}\label{lma:connected_full}
Let $(\X,\sfd,\mm)$ be a PI space and let
  $E$ be a closed bounded $BV$-extension set. Then, there exists $N$, depending only $E$ and the constants on the doubling property and and the Poincar\'e inequality  such that the following holds.
  For every finite Borel partition $\{E_i\}_{i=1}^m$ of $E$ satisfying
  $\mm(E_i)>0$ and $\Pe_E(E_i,E)=0$ for all $1\le i \le m$  we have $m \le N$.
\end{lemma}

\begin{proof}
 Let us denote the extension operator by $T$. Let $r = \diam(E)+1$ and $C,s$ be as in Proposition \ref{prop:Sobo}. We prove that the result holds with the choice 
 \[ N:=\lfloor \mm(E)\mm(B(x,r))^{-1}\max\{ 2 \|T\|,C \|T\|^s \diam(E)^s \} \rfloor. \]

  Let $F \subset E$ be so that $\mm(F)>0$ and $|D\chi_F|_{E}(E) = {\rm P}_E(F,E) =0$.  
 Let $x \in F$. The claim will follow if we prove
\begin{equation}
\label{eq:bound_measure_connected_component}
\frac{\mm(B(x,r))}{\mm(F)} \le \max\left(2\|T\|,C\|T\|^s\diam(E)^s\right).
\end{equation}
 Notice that we have 
 \[
\|\chi_F\|_{BV(E)} = \mm(F).
\]

We suppose that $\mm(F) < \frac{\mm(B(x,r))}{2\|T\|}$. Then
\[
 \int_{B(x,r)}T\chi_F\,\d\mm \le \|T\chi_F\|_{BV(\X)} \le \|T\|\|\chi_F\|_{BV(E)} = \|T\|\mm(F) < \frac{\mm(B(x,r))}{2}.
\]
We define $u:=T\chi_F$ and the previous computation gives that,
\begin{equation*}
    |u(z)-u_{B(x,r)}|=1- u_{B(x,r)} > \frac{1}{2},\qquad \text{for every }z \in F.
\end{equation*}
Combining this observation and the second claim in Proposition \ref{prop:Sobo}, we then get
\begin{align*}
\mm(F) & \le \mm\left(\left\{z \in B(x,r)\,:\,|u(z) - u_{B(x,r)}|>\frac12\right\}\right)\\
& \le C\mm(B(x,r))^\frac{1}{1-s}\left(r|D(T\chi_F)|(\X)\right)^{\frac{s}{s-1}}\\
& \le C\mm(B(x,r))^\frac{1}{1-s}\left(r\|T\|\mm(F)\right)^{\frac{s}{s-1}}.
\end{align*}
Rearranging the terms and raising to power $s-1$ gives
\[
\mm(F) \ge \frac{\mm(B(x,r))}{C\|T\|^sr^s}.
\]
This proves \eqref{eq:bound_measure_connected_component}. 
%
\end{proof}

Let us end this subsection with an example showing that Lemma \ref{lma:connected_full} fails for general metric measure spaces.

\begin{example}
 Lemma \ref{lma:connected_full} is not true if we remove the assumption of the metric measure space being a PI space.
 Let $\X = \mathbb R$ with the Euclidean distance and $\mm=\mathcal{L}^1\restr{C}$, where $C$ is a Cantor set with $\mathcal{L}^1(C)>0$. Then $C$ is a bounded, closed $BV$-extension set that does not have a decomposition into indecomposable sets with respect to $C$.


\end{example}


\section{Compact $BV$-extension sets in the Euclidean plane}\label{sec:plane}

In this section we study the special case of closed subsets of the Euclidean plane. We show that compact $BV$-extension sets of $\mathbb R^2$ with finitely many complementary components are $W^{1,1}$-extension sets. The basic idea is to follow \cite{BR21} and modify the $BV$-extension so that it has zero variation on the boundaries of the complementary components. By the coarea formula, this is reduces to modifying sets of finite perimeter, and by the following decomposition theorem from \cite[Corollary 1]{ALC01} it is then reduced to modifying Jordan loops.

\begin{theorem}[Structure of the essential boundary of sets of finite perimeter in the plane]
\label{thm:planardecomposition}
Let $E \subset \mathbb R^2$ have finite perimeter. Then, there exists a unique decomposition of $\partial^ME$ into rectifiable Jordan curves $\{J_i^+, J_k^-\,:\,i,k \in \mathbb N\}$, up to $\mathcal H^1$-measure zero sets, such that
\begin{enumerate}
    \item Given $\text{int}(J_i^+)$, $\text{int}(J_k^+)$, $i \ne k$, they are either disjoint or one is contained
in the other; similarly, for $\text{int}(J_i^-)$, $\text{int}(J_k^-)$, $i \ne k$, they are either disjoint or one is
contained in the other. Each $\text{int}(J_i^-)$ is contained in one of the $\text{int}(J_k^+)$.
    \item $P(E,\R^2) = \sum_{i}\mathcal H^1(J_i^+) + \sum_k \mathcal H^1(J_k^-)$.
    \item If $\text{int}(J_i^+) \subset \text{int}(J_j^+)$, $i \ne j$, then there is some rectifiable Jordan curve  $J_k^-$ such that $\text{int}(J_i^+)\subset \text{int}(J_k^-) \subset \text{int}(J_j^+)$. Similarly, if $\text{int}(J_i^-) \subset \text{int}(J_j^-)$, $i \ne j$, then there is some rectifiable Jordan curve  $J_k^+$ such that $\text{int}(J_i^-)\subset \text{int}(J_k^+) \subset \text{int}(J_j^-)$.
    \item Setting $L_j =\{i \,:\, \text{int}(J_i^-)\subset \text{int}(J_j^+)\}$ the sets $Y_j = \text{int}(J_j^+) \setminus \bigcup_{i \in L_j}\text{int}(J_i^-)$ are pairwise disjoint, indecomposable and $E = \bigcup_j Y_j$.
\end{enumerate}
\end{theorem}

In order to modify the Jordan loops, we will use the following lemma from \cite{BR21}.

\begin{lemma}[{\cite[Lemma 5.3]{BR21}}]\label{lma:quasi_int}
 Let $\Omega$ be a Jordan domain. For every $x,y \in \overline{\Omega}$, every $\varepsilon > 0$ and any rectifiable curve $\gamma \subset \overline{\Omega}$ joining $x$ to $y$ there exists a curve $\sigma \subset \Omega \cup \{x,y\}$ joining $x$ to $y$ so that
 \[
   \ell(\sigma) \le \ell(\gamma) + \varepsilon.
 \]
\end{lemma}

\subsection{Quasiconvexity of the complement of extension sets}

A main ingredient in modifying the Jordan loops is the quasiconvexity of the connected components of the complement of an extension set. In the case of the homogeneous norm, we have the following result. However, since we are mostly interested in $BV$-extensions with the full norm, we will skip the proof of the lemma and concentrate on the more complicated case presented in Lemma \ref{lma:quasiconvexity_full}.

\begin{lemma}\label{lma:quasiconvexity_hom}
 If $E \subset \mathbb R^2$ is a $\ca{BV}$-extension set, then each connected component of $\mathbb R^2 \setminus \overline{E^1}$ is quasiconvex.
\end{lemma}

%


Unlike for a closed $\ca{BV}$-extension set, the complementary domains of a closed $BV$-extension set need not be quasiconvex. However, in the next lemma we show that they are locally quasiconvex, when the locality is seen with respect to the internal distance. It is interesting to notice that we prove Lemma \ref{lma:quasiconvexity_full}
first for points in $\Omega$ and then the properties (1) and (2) pass (via (3)) to the points on $\overline{\Omega}$. In general, the property (1) for points in $\Omega$ does not imply it for $\overline{\Omega}$. This is seen by taking $\Omega$ to be an infinitely long spiral. Similarly, the property (2) for points in $\Omega$ does not imply it for $\overline{\Omega}$. This is seen by considering a topologist's comb.

\begin{lemma}\label{lma:quasiconvexity_full}
 If $E \subset \mathbb R^2$ is a compact $BV$-extension set, then 
 there exist constants $\delta,C>0$ so that the following four properties hold for all connected components $\Omega$ of $\mathbb{R}^2 \setminus E$ and all $x,y \in \overline{\Omega}$.
 \begin{enumerate}
     \item $\sfd_{\Omega}(x,y) \le C + \|x-y\|$.
     \item If $\sfd_{\Omega}(x,y)< \delta$, then $\sfd_{\Omega}(x,y) \le C\|x-y\|$.
    \item $(\partial \Omega,\sfd_\Omega)$ is totally bounded. 
    \item For $x \in \Omega$ and $0 < r < \delta/2$ the closed ball $\overline{B}_{\sfd_\Omega}(x,r)$ is a closed subset of $\mathbb R^2$.
 \end{enumerate}
\end{lemma}
\begin{proof}
  We have that $E$ has the extension property for sets of finite perimeter with the full norm by \cite[Proposition 3.4]{CKR23} with $C_{\rm Per} \le \|T\|_{BV}$, where $T \colon BV(E) \to \mathbb{R}^2$ is the $BV$-extension operator and $C_{\rm Per}$ is the constant for the perimeter extension property of $E$ with the full norm.
  Let us define
  \[
  \delta = \frac{1}{2C_{\rm Per}\pi}.
  \]

  {\color{blue}{\sc Step 1:}}
  We first show (1) and (2) for points in $\Omega$.
  
  Thus, let $x,y \in \Omega$ be given. If $[x,y] \cap E = \emptyset$, we have $\sfd_{\Omega}(x,y) = \|x-y\|$ and thus (1) and (2) hold. Hence, in the rest of the proof of Step 1, we may assume that $[x,y] \cap E \ne  \emptyset$.
  Every connected component $\Omega$ is open, hence locally rectifiably path-connected. Therefore, since $\Omega$ is path-connected and it is locally rectifiably path-connected, it is rectifiably path-connected \cite[Lemma 2.2]{CaputoCavallucci2024II}.
  This implies that, if $x,y \in \Omega$, then $\sfd_{\Omega}(x,y)<\infty$.
  
  Let $\varepsilon > 0$.
  We denote by $\Gamma$ the collection of curves $\gamma \colon [0,1] \to \Omega$ such that $\gamma(0) = x$ and $\gamma(1) = y$.
  By taking a shorter curve $\gamma$ in $\Gamma$ if necessary, we may assume that $\gamma$ is injective and that  \begin{equation}\label{eq:almostminimalgamma}
      \ell(\gamma) \le \sfd_\Omega(x,y) + \varepsilon.
  \end{equation}   
  Let us consider separately the at most countably many subcurves $\gamma_i \colon [0,1] \to \mathbb R^2$ of $\gamma$ so that $\gamma_i\cap [x,y] = \{\gamma_i(0),\gamma_i(1)\}$. Now, $\alpha_i = \gamma_i \cup [\gamma_i(1),\gamma_i(0)]$ forms a Jordan loop. Let $A_i$ be the bounded component of $\mathbb R^2 \setminus \alpha_i$; see Figure \ref{fig:Modification_in_Omega}. 
  \begin{figure}
    \centering   \includegraphics[width=0.6\linewidth]{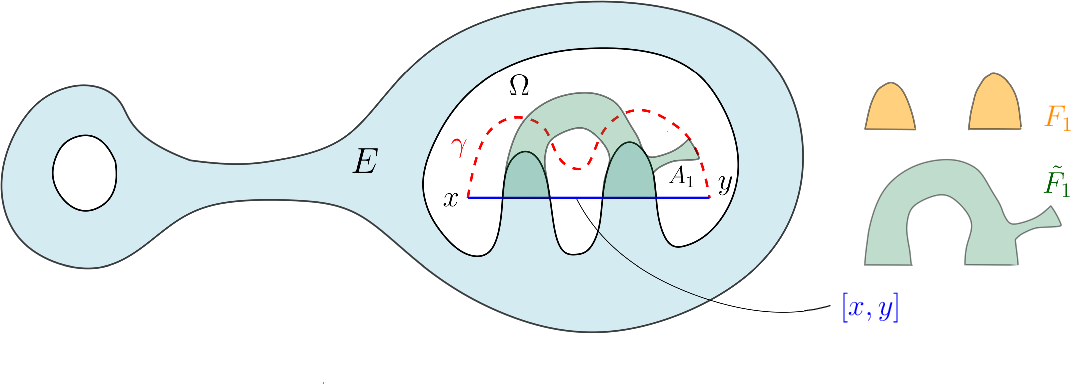}
    \caption{The picture gives a qualitative description of the construction in Step 1. The region $A_1$ is enclosed by the dashed path $\gamma$ and the  line segment $[x,y]$. The picture follows the construction in the case $A_1 \cap E \neq \emptyset$.}
    \label{fig:Modification_in_Omega}
\end{figure}

  If $A_i\cap E = \emptyset$, by Lemma \ref{lma:quasi_int} there exists a curve $\beta_i$ in $A_i\cup \{\gamma_i(0),\gamma_i(1)\}$ joining $\gamma_i(0)$ and $\gamma_i(1)$ with $\ell(\beta_i) \le \|\gamma_i(0)-\gamma_i(1)\| + 2^{-i}\varepsilon$.

  Suppose then that $F_i := A_i\cap E \ne \emptyset$.
  Since $\mathcal{H}^1(\partial A_i)=
  \ell(\alpha_i)<\infty$, we have that $A_i$ has finite perimeter, thus also $F_i$ has finite perimeter.
  Let us consider the minimal perimeter extension $\tilde F_i$ of $F_i$ with respect to the perimeter. We then have
  \begin{equation}\label{eq:F_i_estimate}
   \Pe(\tilde{F}_i) \le  \mathcal{L}^2(\tilde{F}_i) + \Pe(\tilde{F}_i) \le C_{\rm Per} (\mathcal{L}^2(F_i)+ \Pe(F_i,E)).    
  \end{equation}
  By the minimality of $\tilde F_i$ we have
 \begin{equation}\label{eq:F_i_boundary}
 \partial^M\tilde F_i\cup [\gamma_i(0),\gamma_i(1)] = \partial\tilde F_i \cup [\gamma_i(0),\gamma_i(1)].
 \end{equation}
 
  
  We claim that there exists an injective curve
  \begin{equation}
   \alpha_i \subset C_i := \left(\partial\tilde F_i \cup [\gamma_i(0),\gamma_i(1)]\right)
   \cap \left(\partial(A_i\triangle \tilde F_i) \cup (\partial \tilde F_i \cap \gamma_i) \right)
  \end{equation}
  connecting  $\gamma_i(0)$ to $\gamma_i(1)$.
  In order to see this, notice first that by \eqref{eq:F_i_estimate} and \eqref{eq:F_i_boundary} the set $C_i$ has finite 1-dimensional Hausdorff measure. Thus, it is enough to prove that $\gamma_i(0)$ and $\gamma_i(1)$ are in the same connected component of $C_i$. Suppose that this is not the case. 
  Then there exists a Jordan loop $\Gamma \subset \R^2 \setminus C_i$ so that $\gamma_i(0)$ is in the bounded component of $\R^2 \setminus \Gamma$ and $\gamma_i(1)$ in the unbounded component.
  By moving $\Gamma$ slightly, if necessary, $\Gamma \cap [\gamma_i(0),\gamma_i(1)]$ is finite. Moreover, since by assumption each point $z \in \Gamma \cap [\gamma_i(0),\gamma_i(1)]$
  is not in $\partial(A_i \Delta \tilde F_i)$, again by moving $\Gamma$ slightly we may assume that for these points $z$ we have $\Gamma \cap B(z,r)\cap \tilde F^i \ne \emptyset \ne \Gamma \cap B(z,r)\setminus \tilde F^i$ for every $r > 0$.
  Since $\Gamma \cap [\gamma_i(0),\gamma_i(1)]$ is odd, there exists a point $w \in \partial \tilde F^i \setminus [\gamma_i(0),\gamma_i(1)]$. However, if $w \in \gamma_i$, we have a contradiction as then $w \in C_i$. Then, if $w \notin \gamma_i$, we have $w \in \partial (A_i \Delta \tilde F_i)$ and again $w \in C_i$, giving a contradiction. This proves the claim and the existence of $\alpha_i$.
    

  
  
  
  
  We have
  \[
  \ell(\alpha_i) \le \|\gamma_i(0)-\gamma_i(1)\| + \Pe(\tilde F_i)
  \]
  due to \eqref{eq:F_i_estimate} and \eqref{eq:F_i_boundary}.
  Using $\alpha_i$, we construct a new curve $\beta_i$ in $\Omega$ connecting $\gamma_i(0)$ to $\gamma_i(1)$; see Figure \ref{fig:beta_i}.
  \begin{figure}
    \centering
\includegraphics[width=0.6\linewidth]{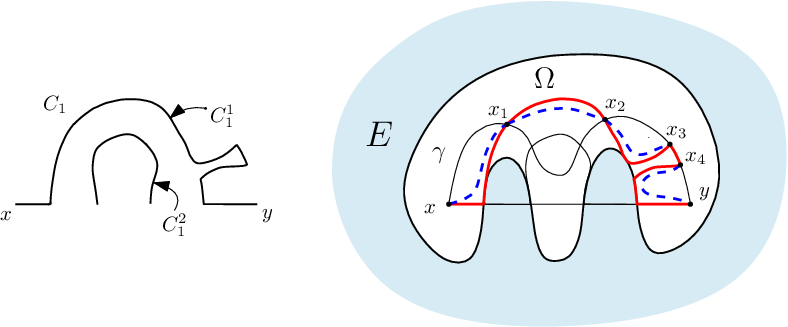}
    \caption{This is a qualitative description of the construction of \(C_1\), \(\alpha_1\) and \(\beta_1\) for the example in Figure \ref{fig:Modification_in_Omega}. On the left-hand side, we have the set \(C_1\), which in this case consists of two connected components, \(C_1^1\) and \(C^2_1\). Then, in the example, the image of the curve \(\alpha_1\) is exactly \(C_1^1\).
    Starting from \(\alpha_1\), by making use of Lemma \ref{lma:quasi_int}, we build \(\beta_1\) by replacing the four subcurves of \(\alpha_1\) connecting the points \(x\) to \(x_1\), \(x_1\) to \(x_2\), \(x_2\) to \(x_3\), and \(x_4\) to \(y\) by the dashed curves connecting the same points, while the subcurve of \(\alpha_1\) connecting \(x_3\) to \(x_4\) remains unchanged. 
    }
\label{fig:beta_i}
\end{figure}
  This is done by taking each maximal subcurve $\tilde \beta$ of $\alpha_i$ for which at most the endpoints intersect $\gamma_i$ and replacing it by a curve given by Lemma \ref{lma:quasi_int} when used in the connected component of $A_i\triangle \tilde F_i$ to whose boundary $\tilde \beta$ belongs to. (Notice that the connected component of $A_i\triangle \tilde F_i$ need not be a Jordan domain, but by removing part of it one can make it into one without changing the fact that the Jordan curve $\tilde \beta$ is on its boundary.) The remainder of the curve $\beta_i$ is defined as $\alpha_i$. This part is contained in $\gamma_i\cap \partial^M\tilde F_i$.
  
  Lemma \ref{lma:quasi_int} then gives
  \[
  \ell(\beta_i) \le \|\gamma_i(0)-\gamma_i(1)\| + \Pe(\tilde F_i) + 2^{-i}\varepsilon. 
  \]  

  We have $[0,1] = L \cup \bigcup_{i \in I} [a_i,b_i]$, where $I$ is countable, $\gamma\restr{L} \subset [0,1](y-x)$ and $\gamma\restr{[a_i,b_i]} = \gamma_i \left( \frac{t-a_i}{b_i-a_i}
  \right)$ for every $i \in I$.
  We then define the final curve $\tilde \gamma \colon [0,1] \to \Omega$ as
  \[
  \tilde\gamma(t) = \begin{cases}
    \beta_i\left(\frac{t-a_i}{b_i-a_i} \right),& \text{ if }t \in [a_i,b_i],\\
    \gamma(t),& t \in L.
   \end{cases}
  \]

  Combining the estimates for $\ell(\beta_i)$, we have
  \[
   \ell(\tilde\gamma) \le \|x-y\| + \sum_i \Pe(\tilde F_i) + \varepsilon.
  \]
  Since $E$ is a $BV$-extension set, we have
  \begin{equation}\label{eq:oneloopestimate}
    \mathcal L^2(\tilde F_i) + \Pe(\tilde F_i) \le C_{\rm Per}(\mathcal L^2( F_i) + \Pe(F_i,E) ).
  \end{equation}
  It follows from the definition of the sets $A_i$ that $A_i \cap A_j =\emptyset$ if $i \neq j$. Therefore, by the definition of $F_i$, we have $F_i \cap F_j =\emptyset$ if $i \neq j$. This implies $\sum_i \mathcal{L}^2(F_i) \le \mathcal{L}^2(E)$.
  
  Moreover, since $\gamma_i \subset \Omega$, we have $\partial^M F_i \cap E \subset [\gamma_i(0),\gamma_1(1)]$, and thus
  \begin{equation}
  \label{eq:bound_PeFiE}
      \Pe(F_i,E) = \mathcal{H}^1 (\partial^M F_i \cap E) \le \|\gamma_i(0)-\gamma_i(1)\|.
  \end{equation}
  
  Towards showing (1), the right-hand side of \eqref{eq:oneloopestimate} can be summed over $i$ and estimated as
  \begin{equation}
  \label{eq:sum_bv_norm_estimated_by_segment}
  \sum_i(\mathcal L^2( F_i) + \Pe(F_i,E)) \stackrel{\eqref{eq:bound_PeFiE}}{\le} \mathcal L^2(E) + \sum_i\|\gamma_i(0)-\gamma_i(1)\| \le \mathcal L^2(E) + \|x-y\|.
  \end{equation}
  This gives
  \begin{equation}
  \label{eq:estimate_internal_distance}
      \begin{aligned}
   \sfd_{\Omega}(x,y) & \le \ell(\tilde\gamma) \le \|x-y\| + \sum_i \Pe(\tilde F_i) + \varepsilon\\
   & \le \|x-y\| + \sum_i\left(\mathcal L^2(\tilde F_i) + \Pe(\tilde F_i)\right) + \varepsilon\\
   & \stackrel{\eqref{eq:oneloopestimate}}{\le} \|x-y\| + C_{\rm Per}\sum_i(\mathcal L^2( F_i) + \Pe(F_i,E)) + \varepsilon\\ 
   & \stackrel{\eqref{eq:sum_bv_norm_estimated_by_segment}}{\le} (1+C_{\rm Per})\|x-y\| + C_{\rm Per}\mathcal L^2(E) + \varepsilon.
  \end{aligned}
  \end{equation}
  We distinguish two cases. If $\dist(x,E) \le 1$ and $\dist(y,E) \le 1$, we can continue the previous estimate and we have, by triangular inequality,
    \begin{equation}\label{eq:close}
        \sfd_{\Omega}(x,y) \le \|x-y\| + C_{\rm Per}(2+{\rm diam}(E)) + C_{\rm Per}\mathcal L^2(E) + \varepsilon.
    \end{equation}
    In the case where at least one of the inequalities $\dist(x,E) > 1$ and $\dist(y,E) > 1$ holds, we can argue as follows. We can first travel along the segment $[x,y]$ first closer to $E$ and thus we select two points $\tilde{x},\tilde{y} \in \Omega$ such that $[x,\tilde{x}] \subset \Omega$, $[\tilde{y},y] \subset \Omega$, $\dist(\tilde{x},E) \le 1$, 
    and $\dist(\tilde{y},E) \le 1$. This is possible due to the assumption $[x,y] \cap E \ne \emptyset$.
    The estimate \eqref{eq:close} holds for the points $\tilde{x},\tilde{y}$, and so
    \begin{equation*}
    \begin{aligned}
     \sfd_{\Omega}(x,y) & \le \sfd_{\Omega}(x,\tilde x) + \sfd_{\Omega}(\tilde x,\tilde y) + \sfd_{\Omega}(\tilde y,y)\\
     & \le \|x-\tilde x\|  + \|\tilde x-\tilde y\| + C_{\rm Per}(2+{\rm diam}(E)) + C_{\rm Per}\mathcal L^2(E) + \varepsilon
     + \|\tilde y-y\|\\
     & = \|x-y\| + C_{\rm Per}(2+{\rm diam}(E)) + C_{\rm Per}\mathcal L^2(E) + \varepsilon.
    \end{aligned}
    \end{equation*}

  We now prove (2). 
  We fix $\delta:=(4 C_{\rm Per}\pi)^{-1}$ and let $0 <\varepsilon \le \delta$.
  Since in this case we assume that $\sfd_\Omega(x,y)< \delta$, we have that $\ell(\gamma) \le \delta +\varepsilon$.
  By the definition of the sets $F_i$'s, we have that $\cup_i F_i \subset B(0,\ell(\gamma))$. In particular, using again that the sets $\{F_i\}_i$ are pairwise disjoint, we have 
  \[
  \sum_i(\mathcal L^2( F_i) + \Pe(F_i,E)) \le \pi\ell(\gamma)^2 + \sum_i\|\gamma_i(0)-\gamma_i(1)\| \le 2 \pi\delta\ell(\gamma) + \|x-y\|.
  \]
  To conclude the estimate, by using the previous inequality, we have
  \begin{align*}
     \sfd_{\Omega}(x,y) 
   & \le \|x-y\| + C_{\rm Per}\sum_i(\mathcal L^2( F_i) + \Pe(F_i,E))+ \varepsilon\\ 
   & \le (1+C_{\rm Per})\|x-y\| + 2 C_{\rm Per} \pi\delta\ell(\gamma) + \varepsilon\\
   & \le (1+C_{\rm Per})\|x-y\| + \frac{1}{2}\ell(\gamma) + \varepsilon\\
   & \le (1+C_{\rm Per})\|x-y\| + \frac{1}{2}\sfd_{\Omega}(x,y) + 2\varepsilon,
  \end{align*}
  and taking the limit as $\varepsilon \to 0$ we obtain
  \[
   \sfd_{\Omega}(x,y) \le 2(1+C_{\rm Per})\|x-y\| \le {4 C_{\rm Per}}\|x-y\|.
  \]


 {\color{blue}{\sc Step 2:}} Let $R>0$ be large enough so that $E \subset B(0,R)$. We prove that $(\Omega \cap B(0,R),\sfd_\Omega)$ is totally bounded.
 
 Take $0 <\varepsilon < \delta/6$ and consider a maximal $\varepsilon$-separated net $\{x_i\}_{i \in I}$ in $\Omega \cap B(0,R)$ with respect to $\sfd_\Omega$. We form a graph $G = (\mathcal{V},\mathcal{E})$ where $\mathcal{V} = \{x_i\}_{i \in I}$ and 
 \begin{equation*}
     \mathcal{E} = \left\{(x_i,x_j) \in \mathcal{V}\times \mathcal{V}:\,x_i \neq x_j \text{ and such that }\sfd_\Omega(x_i,x_j) \le 3\varepsilon\right\}.
 \end{equation*}
 Our aim is to show that $G$ is a finite graph.
 Since $(\Omega\cap B(0,R),\sfd_\Omega)$ is a length space, $G$ is connected.

 Let $x_k,x_j \in \mathcal{V}$ with $k \neq j$. By (1), there exists an injective curve $\gamma \subset \Omega$ connecting $x_k$ to $x_j$ with $\ell(\gamma) \le C + 2R$. Let $\{y_i\}_{i=1}^N$ be a collection of points in $\gamma$ so that $x_k = y_1$, $x_j = y_N$, $\sfd_\Omega(y_{i},y_{i+1})\le \varepsilon$, and $N-1 \le \ell(\gamma)/\varepsilon$.
 Since $\{x_i\}$ is a maximal $\varepsilon$-separated net, for every $y_i$ there exists $x_{n_i} \in \mathcal{V}$ with $\sfd_\Omega(y_i,x_{n_i})<\varepsilon$. Then,
 \[
  \sfd_\Omega(x_{n_i},x_{n_{i+1}}) \le \sfd_\Omega(y_i,x_{n_i}) + \sfd_\Omega(y_i,y_{i+1}) + \sfd_\Omega(y_{i+1},x_{n_{i+1}}) \le 3\varepsilon
 \]
 and so $\{x_{n_i},x_{n_{i+1}}\} \in \mathcal{E}$. This means that the distance between $x_k$ and $x_j$ is at most $N-1$ in the graph. Since the upper bound on $N$ is independent of $k$ and $j$, the diameter of the graph $G$ is at most $\frac{1}{\varepsilon}(C+2R)$.
 
 Condition (2) now implies that if $\{x_i,x_j\},\{x_i,x_k\} \in \mathcal{E}$, then, since $\sfd_\Omega(x_j,x_k) \le \sfd_\Omega(x_i,x_j) + \sfd_\Omega(x_i,x_k)\le 6\varepsilon < \delta$
\[
\|x_j-x_k\| \ge \frac{1}{C}\sfd_\Omega(x_j,x_k) \ge \frac{\varepsilon}{C},
\]
while we also have
\[
 \|x_j-x_i\| \le \sfd_\Omega(x_j,x_i) \le 3\varepsilon.
\]
 Thus, a simple volume argument gives that the degree of any vertex $x_i$ is bounded by a constant.

 The bounds on the diameter and degree of $G$ imply a bound on the size of $G$. In particular, $G$ is finite, which in turn implies that $(\Omega \cap B(0,R),\sfd_\Omega)$ is totally bounded.

 {\color{blue}{\sc Step 3:}}
 We prove that (1) and (2) hold for points that belongs to $\overline{\Omega}$.
 
 Let $x,y \in \overline{\Omega}$ and take sequences $(x_i),(y_i) \subset \Omega$ with $\|x_i-x\|\le 2^{-i}$ and $\|y_i-y\| \le 2^{-i}$ for all $i$. By Step 2, $\Omega \cap B(0,R)$ can be covered by a finite number of balls, computed with respect to the distance $\sfd_\Omega$, with radius less or equal than $\delta/2$. Hence, there exist subsequences of $(x_i)$ and $(y_i)$ so that $\sfd_\Omega(x_i,x_j) < \delta$ and $\sfd_\Omega(y_i,y_j) < \delta$ for all $i,j$. Now, the rest of the argument is a standard concatenation of curves using (2) in $\Omega$. Let us still write the argument for the convenience of the reader. 
 
Let us first check (1) for $x,y \in \overline{\Omega}$. For each $i\in \mathbb{N}$, let $\gamma_{x,i}$ be a curve in $\Omega$ joining $x_{i+1}$ to $x_i$ with $\ell(\gamma_{x,i}) \le C\|x_{i+1}-x_i\| \le C2^{-i+1}$ given by (2). Similarly, for each $i\in \mathbb{N}$, let $\gamma_{y,i}$ be a curve in $\Omega$ joining $y_{i}$ to $y_{i+1}$ with $\ell(\gamma_{y,i}) \le C2^{-i+1}$. Then, for $j \ge 3$, define $\gamma_j \colon [0,1] \to \Omega\cup\{x,y\}$ as
 $\gamma_{x,i}$ on $[2^{-i},2^{-i+1}]$ and $\gamma_{y,i}$ on $[1-2^{-i+1},1-2^{-i}]$
 for all $i \ge j$ and let $\gamma_j$ on $[2^{-j+1},1-2^{-j+1}]$ be defined as the curve in $\Omega$ joining $x_{j}$ to $y_j$ with length less than $\sfd_\Omega(x_j,y_j)+2^{-j}$. Moreover, we define $\gamma_j(0):=x$ and $\gamma_j(1):=y$ and notice that $\gamma \in C([0,1],\mathbb{R}^2)$.
By using (1) for $x_j,y_j$, we can estimate 
\begin{align*} \sfd_\Omega(x,y)\le \ell(\gamma_j) &\le \sum_{i=j}^\infty \left(\ell(\gamma_{x,i}) + \ell(\gamma_{y,i})\right)
+ \sfd_\Omega(x_j,y_j)+2^{-j} \\
 &\le \sfd_\Omega(x_j,y_j) + C2^{-j+4} \le C + \|x_j-y_j\| + C2^{-j+4}\\
 & \le C + \|x-y\| + \left(\|x-x_j\| + \|y-y_j\| + C2^{-j+4}\right),
\end{align*}
 where 
 \[
 \|x-x_j\| + \|y-y_j\| + C2^{-j+4} \to 0
 \]
 as $j \to \infty$. This shows that (1) holds. 
 
 For (2) we assume that $\sfd_{\Omega}(x,y)< \delta$ holds and let $\gamma \colon[0,1] \to \Omega\cup\{x,y\}$ with $\ell(\gamma) < \delta$. Then, for $0<\varepsilon < \frac{1}{2}(\delta - \ell(\gamma))$ we can take $\tilde x,\tilde y \in \gamma \cap \Omega$ with $\sfd_\Omega(x,\tilde x)< \varepsilon$ and  $\sfd_\Omega(y,\tilde y)< \varepsilon$. We have
 $\sfd_\Omega(\tilde x, \tilde y) < \delta$ and so by (2) we get 
 \[
 \sfd_{\Omega}(\tilde x,\tilde y) \le C\|\tilde x-\tilde y\| \le C\|x-y\| + 2C\varepsilon.
 \]
 Thus,
 \[
  \sfd_{\Omega}(x,y) \le  \sfd_{\Omega}(\tilde x,\tilde y) + 2\varepsilon \le C\|x-y\| + 2(C+1)\varepsilon,
 \]
 giving (2) as $\varepsilon \to 0$.
 
 {\color{blue}{\sc Step 4:}}
 We prove that $(\bar{\Omega} \cap B(0,R),\sfd_\Omega)$ is totally bounded for every $R>0$. This in particular implies, since $E \subset B(0,R)$ for $R$ sufficiently large, that $(\partial \Omega,\sfd_\Omega)$ is totally bounded.
 It is enough to prove that for every point $x\in \bar{\Omega} \cap B(0,R)$ there exists a point $x_{i_0} \in\{x_i\}_i$, where the points $\mathcal{V}=\{x_i\}_i$ are given in Step 2, such that $\sfd_\Omega(x,x_{i_0}) < 2\varepsilon$.
 Since $\{x_i\}_i$ is an $\varepsilon$-separated net for $\Omega \cap B(0,R)$ the previous statement is satisfied for points belonging to this set.
 Therefore, fix a point $x \in \partial \Omega$. 
 We consider an arbitrary point $y \in \Omega$ and by Step 3, we have that $\sfd_\Omega(x,y)<\infty$. Consider one rectifiable path $\gamma$ for the definition of $\sfd_\Omega$ and we can find $p_x \in \gamma \cap \Omega$ such that $\sfd_{\Omega}(x,p_x)<\varepsilon$. By the definition of the set $\{x_i\}_i$, there exists $\{x_{i_0}\}$ such that $\sfd_\Omega(p_x,x_{i_0}) < \varepsilon$. By applying the triangular inequality, we conclude.

 {\color{blue}{\sc Step 5:}} Finally, (4) follows by taking a converging (in the Euclidean topology) sequence $(y_i) \subset \overline{B}_{\sfd_\Omega}(x,r)$. Let $y$ be the limit of $(y_i)$. By taking a subsequence of $(y_i)$ we may assume $\|y-y_i\| \le 2^{-i}$. By using (2) between $y_i$ and $y_{i+1}$ and by concatenating curves similarly to the beginning of Step 3, we get that
 \[
 \sfd_\Omega(y,y_i) \le C2^{-i+1}
 \]
 for all $i \in \mathbb N$. Thus,
 \[
 \sfd_\Omega(x,y) \le r + C2^{-i+1}
 \]
 for all $i \in \mathbb N$ and so
 $y \in \overline{B}_{\sfd_\Omega}(x,r)$.
\end{proof}

\begin{remark}[On the local quasiconvexity constant]
    We point out that, by looking at the estimate on the last lines of Step 1 of the proof, we see that if $\sfd_\Omega(x,y)< \delta$, then $\sfd_\Omega(x,y)\le 4 \|T\|_{BV} \|x-y\|$. So the local quasiconvexity constant of the complement of $E$ depends only on the norm of the extension operator $T \colon BV(E) \to BV(\mathbb{R}^2)$.
\end{remark}






Lemma \ref{lma:quasiconvexity_full} above helps in getting the following conclusion from Lemma \ref{lma:connected_full}. We will need it later to prove Proposition \ref{prop:kicking_for_BV}.

\begin{proposition}\label{lma:small_intersect}
Let $E \subset \mathbb R^2$ be a bounded $BV$-extension set. Let
$\{\Omega_i\}_{i \in I}$ be the connected components of \(\R^2\setminus \overline{E^1}\). Then
$\partial \Omega_i \cap \partial \Omega_j$ is finite for every $i \ne j$.
\end{proposition}
\begin{proof}
Suppose that there exist $i \neq j$ so that
$\partial \Omega_i \cap \partial \Omega_j$ is infinite.
By Lemma \ref{lma:quasiconvexity_full}, for any two $x_1,x_2 \in \partial \Omega_i \cap \partial \Omega_j$, $x_1 \ne x_2$, there exist injective curves $\gamma_i \subset \Omega_i\cup\{x_1,x_2\}$ and $\gamma_j \subset \Omega_j\cup\{x_1,x_2\}$ connecting $x_1$ to $x_2$. Now, $\gamma_i\cup \gamma_j$ forms a Jordan loop. Let $U_1,U_2$ be the two connected components of $\mathbb R^2 \setminus (\gamma_i\cup\gamma_j)$. Then, we claim that $\mathcal L^2(\overline{E^1}\cap U_k)>0$ and $\Pe_{\overline{E^1}}(\overline{E^1}\cap U_k, \overline{E^1})=0$ for both $k\in\{1,2\}$.

We prove that $\mathcal{L}^2(\overline{E^1}\cap U_1)>0$. By contradiction, assume that $\mathcal{L}^2(\overline{E^1}\cap U_1)=0$. If $\overline{E^1}\cap U_1=\emptyset$, then $\Omega_i$ and $\Omega_j$ both belong to the same connected component of $U_1$, that is impossible.
Therefore, we can only have $\overline{E^1}\cap U_1\neq\emptyset$. Let $x \in \overline{E^1}\cap U_1$. Hence, there exists a sequence $\{y_n\}_n \subset E^1 \cap U_1$ converging to $x$. Fix one such $n$ and we have that, for some $r>0$, $B(y_n,r)\subset U_1$ and $\mathcal{L}^2(B(y_n,r) \cap E)>0$. Since $\mm(E \Delta \overline{E^1})=0$, we have $\mathcal{L}^2(B(y_n,r) \cap \overline{E^1})>0$, thus giving a contradiction. A similar argument holds for $U_2$.

We prove that $\Pe_{\overline{E^1}}(\overline{E^1}\cap U_k, \overline{E^1})=0$ for both $k\in\{1,2\}$. We prove it for $U_1$. We have
\begin{equation}
    \label{eq:submodularity}\Pe_{\overline{E^1}}(\overline{E^1}\cap U_1, \overline{E^1}) \le \Pe_{\overline{E^1}}(\overline{E^1}, \overline{E^1}) + \Pe_{\overline{E^1}}(U_1, \overline{E^1}). 
\end{equation}

Notice that $\Pe_{\overline{E^1}}(\overline{E^1}, \overline{E^1})=0$. Moreover, $\Pe_{\overline{E^1}}(U_1, \overline{E^1})\le \Pe(U_1, \overline{E^1}) =\mathcal{H}^1(\partial U_1 \cap \overline{E^1})=0$, where the last equality follows from the fact that the set $\partial U_1 \cap \overline{E^1}$ consists of only two points. This gives that $\Pe_{\overline{E^1}}(\overline{E^1}\cap U_1, \overline{E^1})=0$. The proof for $U_2$ is similar.
 
This means that we can decompose $\overline{E^1}$ into two parts. We then continue by selecting a point $x_3 \in U_k \cap \partial \Omega_i \cap \partial \Omega_j$ for a $k\in \{1,2\}$ for which
$U_k \cap \partial \Omega_i \cap \partial \Omega_j$ is infinite.
Assume it is $U_1$, the other case being similar.
Then, the complement of the Jordan loop as above connecting $x_1$ to $x_3$ consists of two connected components $V_1,V_2$. Thus,  $U_1 \cap \overline{E^1}$ is decomposed into two parts, $U_1 \cap V_1 \cap \overline{E^1}$ and $U_1 \cap V_2 \cap \overline{E^1}$. By a similar argument as above, both sets have positive measure. Moreover, $\Pe_{\overline{E^1}}(\overline{E^1}\cap U_1 \cap V_k,\overline{E^1})=0$ for $k \in \{1,2\}$. We prove for $k=1$, the other case being similar. Indeed,
\begin{equation*}
    \Pe_{\overline{E^1}}(\overline{E^1}\cap U_1 \cap V_1, \overline{E^1}) \le \Pe_{\overline{E^1}}(\overline{E^1}, \overline{E^1}) + \Pe_{\overline{E^1}}(U_1 \cap V_1, \overline{E^1})\le \mathcal{H}^1(\partial U_1 \cap \overline{E^1})+ \mathcal{H}^1(\partial V_1 \cap \overline{E^1})=0.
\end{equation*}
Thus, we define $A_1:=\overline{E^1}\cap (\mathbb{R}^2\setminus U_1)$, $A_2:=\overline{E^1}\cap U_1 \cap V_1$ and $A_3:=\overline{E^1}\cap U_1 \cap (\mathbb{R}^2\setminus V_1)$.
We have $\overline{E^1}=A_1 \cup A_2 \cup A_3$ and
\begin{equation*}
    \mathcal{L}^2(A_i)>0,\,\quad \Pe_{\overline{E^1}}(A_i,\overline{E^1})=0\quad\text{for }i=\{1,2,3\}. 
\end{equation*}

We then continue recursively. Since we can continue this infinitely we obtain a contradiction with Lemma \ref{lma:connected_full} according to which the process should stop in $N$ steps. 
\end{proof}

\begin{remark}
In Lemma \ref{lma:small_intersect}, it is essential to take the connected components of the complement of \(\overline{E^1}\) instead of a general closed set \(E\). 
Indeed, even if Lemma \ref{lma:connected_full} holds both for \(E\) and \(\overline{E^1}\), the conclusion of Lemma \ref{lma:small_intersect} might not be satisfied by the set \(E\). An example is two closed squares linked by two segments, e.g.\ $E:= [-1,0]\times [0,1] \cup [1,2]\times [0,1] \cup [(-1,0),(1,0)] \cup [(-1,1),(1,1)]$.
\end{remark}

\subsection{From $BV$-extension to $W^{1,1}$-extension}

In the case of planar simply connected closed sets we can show that being $BV$-extension set implies being $W^{1,1}$-extension set.

%

%

\begin{lemma}\label{lem:kicking}
Let $E \subset \R^2$ be a compact ${BV}$-extension set. 
Then, there exists a constant $C \ge 1$ such that for any connected component $\Omega$ of $\R^2 \setminus E$, any set of finite perimeter $F \subset \R^2$ and any $\varepsilon>0$ there exists $\tilde{F} \subset \R^2$ such that
\begin{enumerate}
\item $F \setminus  \overline{\Omega} = \tilde{F} \setminus \overline{\Omega}$,
\item 
$\mathcal{H}^1 (\partial^M \tilde{F}\cap \overline{\Omega}) \leq C \mathcal{H}^1(\partial^M F\cap \overline{\Omega})$, 
\item $\mathcal{H}^1 (\partial^M \tilde{F} \cap \partial \Omega) = 0$, 
\item 
$\mathcal L^2(\tilde F \Delta F) \le \varepsilon$.
\end{enumerate}
\end{lemma}
\begin{proof}
We will use Lemma \ref{lma:quasiconvexity_full} to locally implement the ideas from \cite{BR21}. Let $\delta,C>0$ be the constants in the statement of Lemma \ref{lma:quasiconvexity_full}.
Without loss of generality we may assume that $\mathcal L^2(F) > 0$.
For a given $\varepsilon$, we take $r > 0$ such that 
\begin{equation}\label{eq:r_choice}
 r < \frac{\delta}{2}\qquad \text{and}\qquad \mathcal{L}^2(B(E,2Cr)\setminus E) \leq \varepsilon.  
\end{equation}
This can be done since $E$ is compact.
By Lemma \ref{lma:quasiconvexity_full} (3), there exists a cover $\bigl\{\overline{B}_{\sfd_{\Omega}}(x_{k},r) \bigr\}_{k=1}^{N}$ of the intersection of the boundaries $\partial F \cap \partial \Omega$ so that $N < \infty$ and $x_k \in \Omega$ for every $k$.

We proceed by modifying the set $F$ in one enlarged ball $\overline{B}_{\sfd_{\Omega}}(x_j,Cr)$ at a time. Let us set $F_0 = F$.
Suppose then that a set $F_k \subset \mathbb R^2$ of finite perimeter has been defined.
We construct the set $F_{k+1}$ as follows.

 We distinguish two cases. If $\mathcal{H}^1(\partial E \cap \partial^M F_k \cap \overline{B}_{\sfd_\Omega}(x_{k+1},r))=0$, we set $F_{k+1}=F_k$. If $\mathcal{H}^1(\partial E \cap \partial^M F_k \cap \overline{B}_{\sfd_\Omega}(x_{k+1},r))>0$, we argue as follows.
We decompose $\partial^M F_k$ into Jordan curves via Theorem \ref{thm:planardecomposition} so that
\[
\partial^M F_k = \bigcup_n J_n^+ \cup \bigcup_m J_m^-,
\]
where $J_n^+$ denotes the Jordan curves of the external boundary and $J_m^-$ the Jordan curves of the internal boundary.
Parametrize each of the Jordan curves by homeomorphisms (onto their image) $f_n^+ \colon \mathbb{S}^1 \to J_n^+$ and $f_m^- \colon \mathbb{S}^1 \to J_m^-$.

For every $n \in \mathbb{N}$, we consider 
$\overline{B}_{\sfd_{\Omega}}(x_{k+1},r) \cap \partial E \cap \partial J_n^+$. If $\mathcal{H}^1(\overline{B}_{\sfd_{\Omega}}(x_{k+1},r) \cap \partial E \cap \partial J_n^+)=0$, no  modification is needed. If
\[
\mathcal H^1(\overline{B}_{\sfd_{\Omega}}(x_{k+1},r) \cap \partial E \cap \partial J_n^+)>0
\]
we proceed as follows.
Let $\{I_{n,j}\}_j$ be a collection of closed arcs of $\mathbb S^1$ with endpoints $a_{n,j},b_{n,j}$ and disjoint interiors so that
\[
f_n^+(a_{n,j}), f_n^+(b_{n,j}) \in \overline{B}_{\sfd_{\Omega}}(x_{k+1},r),
\]
\[
\ell(f_n^+(I_{n,j})) \le \delta,
\]
\begin{equation}
 \label{eq:third_property}
 \mathcal H^1(f_n^+(I_{n,j})\cap \partial \Omega) \ge \frac12 \mathcal H^1(f_n^+(I_{n,j})),    
\end{equation}

and
\begin{equation}
    \label{eq:fourth_property}
    \mathcal{H}^1(\overline{B}_{\sfd_{\Omega}}(x_{k+1},r) \cap \partial E \cap \partial J_n^+ \setminus \bigcup_jf_n^+(I_{n,j})) = 0.
\end{equation}
Fix $j$. By definition of $f_n^+(a_{n,j}), f_n^+(b_{n,j})$ and the fact that $x_{k+1} \in \Omega$, we have that
\[
\sfd_\Omega(f_n^+(a_{n,j}), f_n^+(b_{n,j})) \le 2r <\delta.
\]
Therefore, using Lemma \ref{lma:quasiconvexity_full} we can take an injective curve $\gamma_{n,j}$ in $\Omega\cup \{f_n^+(a_{n,j}), f_n^+(b_{n,j})\}$ joining $f_n^+(a_{n,j})$ to $f_n^+(b_{n,j})$ and satisfying
\begin{equation}\label{eq:lengthbound1}
\ell(\gamma_{n,j}) \le C \|f_n^+(a_{n,j}) - f_n^+(b_{n,j})\|.
\end{equation}
Let $G_{n,j}$ be the bounded connected component of $\mathbb R^2 \setminus (f_n^+(I_{n,j}) \cup \gamma_{n,j})$.
The set $(a_{n,j},b_{n,j}) \setminus (f_n^+)^{-1}(\overline{B}_{\sfd_{\Omega}}(x_{k+1},r))$ is open and so it consists of countably many open intervals $(a_{n,j,l},b_{n,j,l}) \subset I_{n,j}$.
Let $I_{n,j,l}:=[a_{n,j,l},b_{n,j,l}]$.
For each $l$, let $\gamma_{n,j,l}$ be an injective curve in $\Omega\cup \{f_n^+(a_{n,j,l}), f_n^+(b_{n,j,l})\}$ joining $f_n^+(a_{n,j,l})$ to $f_n^+(b_{n,j,l})$ and satisfying
\begin{equation}\label{eq:lengthbound2}
\ell(\gamma_{n,j,l}) \le C \|f_n^+(a_{n,j,l}) - f_n^+(b_{n,j,l})\|.
\end{equation}
Let $G_{n,j,l}$ be the bounded connected component of $\mathbb R^2 \setminus (f_n^+([a_{n,j,l},b_{n,j,l}]) \cup \gamma_{n,j,l})$.

We repeat the above construction of $G_{n,j}$ and $G_{n,j,l}$ with $f_n^+$ replaced by $f_m^-$ and call the resulting sets
$H_{m,j}$ and $H_{m,j,l}$, respectively.

With these definitions we can then set
\begin{equation}\label{eq:Fk1_def}
 F_{k+1} \coloneqq F_k \cup \left(\bigcup_{n,j} G_{n,j} \setminus \bigcup_{n,j,l} G_{n,j,l}\right) \setminus \left(\bigcup_{m,j} H_{m,j} \setminus \bigcup_{m,j,l} H_{m,j,l}\right).
\end{equation}

Now that we have defined $F_k$ for all $k$, we simply set $\tilde F \coloneqq F_N$. Let us check that $\tilde{F}$ satisfies all the properties in the statement of the lemma.

By the definition \eqref{eq:Fk1_def} of the sets $F_{k+1}$, we have that $F_{k+1} \setminus \overline{\Omega} = F_k \setminus \overline{\Omega}$ for all $k$, because
\begin{equation*}
    \left(\bigcup_{n,j} G_{n,j} \setminus \bigcup_{n,j,l} G_{n,j,l}\right) \setminus \left(\bigcup_{m,j} H_{m,j} \setminus \bigcup_{m,j,l} H_{m,j,l}\right) \subset \bar{\Omega}.
\end{equation*}

Hence, (1) holds. By the upper bounds on the lengths of the curves in \eqref{eq:lengthbound1} and \eqref{eq:lengthbound2}, we have that
\[
F_{k+1}\setminus B(E,2Cr) = F_k \setminus B(E,2Cr)
\]
for all $k$. Hence, $\tilde{F}\setminus B(E,2Cr) = F \setminus B(E,2Cr)$. This implies that $\tilde{F} \Delta F \subset B(E,2Cr)\setminus E$. By the choice \eqref{eq:r_choice} of $r$ the condition (4) holds.

In order to verify (2) and (3), we first observe that 
for every $k$
\begin{equation}\label{eq:changeinF_k}
\begin{split}
\partial^M F_{k+1} \Delta \partial^M F_{k} \subset & \bigcup_{n,j} (f_n^+(I_{n,j}) \cup \gamma_{n,j}) \cup \bigcup_{n,j,l} (f_n^+([a_{n,j,l},b_{n,j,l}]) \cup \gamma_{n,j,l})\\
& \cup \bigcup_{m,j} (f_m^-(I_{m,j}) \cup \gamma_{m,j}) \cup \bigcup_{m,j,l} (f_m^-([a_{m,j,l},b_{m,j,l}]) \cup \gamma_{m,j,l}).
\end{split}
\end{equation}
Moreover, by the construction of $F_{k+1}$, we have
\begin{equation}\label{eq:decreasingboundary}
 \partial^M F_{k+1}\cap \partial \Omega \subset \partial^M F_{k}\cap \partial \Omega 
\end{equation}
 and, as a consequence of \eqref{eq:fourth_property} and the definitions of $G_{n,j},G_{n,j,l},H_{n,j}$, and $H_{n,j,l}$,
\begin{equation}\label{eq:zeroboundary}
 \mathcal H^1(\overline{B}_{\sfd_{\Omega}}(x_{k+1},r)\cap \partial^M F_{k+1}\cap \partial \Omega) =0.
\end{equation}
Combining \eqref{eq:decreasingboundary} and \eqref{eq:zeroboundary} with the fact that $\{\overline{B}_{\sfd_{\Omega}}(x_{k},r)\}_k$ cover $\partial \Omega$, we obtain (3).

We then compute
\begin{equation}
\label{eq:estimate_boundaries_with_quasiconvexity}
\begin{aligned}
    \mathcal{H}^1((f_n^+(I_{n,j}) \cup \gamma_{n,j})\cap \overline{\Omega}) &\le \ell(\gamma_{n,j})+ \mathcal{H}^1(f_n^+(I_{n,j})\cap \overline{\Omega})\\
    & \stackrel{\eqref{eq:lengthbound1}}{\le} C \|f_n^+(a_{n,j}) - f_n^+(b_{n,j})\|
     + \mathcal{H}^1(f_n^+(I_{n,j}))\\
     & \le 2 C \mathcal{H}^1(f_n^+(I_{n,j}).\\
\end{aligned}
\end{equation}
This gives, using \eqref{eq:third_property}
\begin{equation}
\label{eq:estimate_symmetric_difference_comp1}
    \mathcal{H}^1((f_n^+(I_{n,j}) \cup \gamma_{n,j})\cap \overline{\Omega}) \le 4 C \mathcal H^1(f_n^+(I_{n,j})\cap \partial \Omega).
\end{equation}

Similarly, we can repeat the computation in \eqref{eq:estimate_boundaries_with_quasiconvexity} (using \eqref{eq:lengthbound2} instead of \eqref{eq:lengthbound1})
and we have that
\begin{equation*}
    \mathcal{H}^1((f_n^+(I_{n,j,l}) \cup \gamma_{n,j,l})\cap \overline{\Omega}) \le 2 C \mathcal H^1(f_n^+(I_{n,j,l})).
\end{equation*}
By summing up over $l$, we have that
\begin{equation}
\label{eq:estimate_symmetric_difference_comp2}
\begin{split}
    \sum_{l} \mathcal{H}^1((f_n^+(I_{n,j,l}) \cup \gamma_{n,j,l})\cap \overline{\Omega}) & \le 2 C \sum_l \mathcal H^1(f_n^+(I_{n,j,l})) \le 2 C \mathcal H^1(f_n^+(I_{n,j})) \\
    & \stackrel{\eqref{eq:third_property}}{\le} 4C \mathcal{H}^1(f_n^+(I_{n,j}) \cap \partial \Omega).
\end{split}
\end{equation}

We now estimate
\begin{equation}\label{eq:lengthbound}
\begin{aligned}
 \mathcal H^1(\partial^M F_{k+1} \cap \overline{\Omega}) & \le \mathcal H^1(\partial^M F_{k} \cap \overline{\Omega}) + \mathcal H^1( (\partial^M F_{k} \Delta \partial^M F_{k+1}) \cap \overline{\Omega}).\\
 \end{aligned}
 \end{equation}
The second term above can be bounded by using \eqref{eq:changeinF_k} to get
\begin{equation*}
    \begin{aligned}
    \mathcal{H}^1((\partial^M F_{k+1} \Delta \partial^M F_{k}) \cap \overline{\Omega}) & \le  \sum_{n,j} \mathcal{H}^1((f_n^+(I_{n,j}) \cup \gamma_{n,j})\cap \overline{\Omega})\\
    & \quad + \sum_{n,j,l} \mathcal{H}^1((f_n^+([a_{n,j,l},b_{n,j,l}]) \cup \gamma_{n,j,l})\cap \overline{\Omega})\\
    & \quad + \sum_{m,j} \mathcal{H}^1((f_m^-(I_{m,j}) \cup \gamma_{m,j})\cap \overline{\Omega})\\
    & \quad + \sum_{m,j,l} \mathcal{H}^1((f_m^-([a_{m,j,l},b_{m,j,l}]) \cup \gamma_{m,j,l})\cap \overline{\Omega})\\
    &\le 16 C \mathcal{H}^1(\overline{B}_{\sfd_{\Omega}}(x_{k+1},r) \cap \partial \Omega \cap \partial^M F_k),
    \end{aligned}
\end{equation*}
where in the last line we used \eqref{eq:fourth_property}, \eqref{eq:estimate_symmetric_difference_comp1} and \eqref{eq:estimate_symmetric_difference_comp2}.

This gives
\begin{equation}\label{eq:lengthbound_final}
\begin{aligned}
 \mathcal H^1(\partial^M F_{k+1} \cap \overline{\Omega}) & \le \mathcal H^1(\partial^M F_{k} \cap \overline{\Omega}) + 16 C \mathcal{H}^1(\overline{B}_{\sfd_{\Omega}}(x_{k+1},r) \cap \partial \Omega \cap \partial^M F_k)\\
 &= \mathcal H^1(\partial^M F_{k} \cap \overline{\Omega}) + 16 C \mathcal{H}^1(\overline{B}_{\sfd_{\Omega}}(x_{k+1},r)\setminus \cup_{i=1}^{k} \overline{B}_{\sfd_{\Omega}}(x_{i},r) \cap \partial \Omega \cap \partial^M F_k). \\
 \end{aligned}
 \end{equation}
In the last equality, we used that 
\begin{equation*}
    \partial^M F_k \cap \partial \Omega \subset \bigcup_{j={k+1}}^N \overline{B}_{\sfd_{\Omega}}(x_{j},r) \setminus \bigcup_{i=1}^k \overline{B}_{\sfd_{\Omega}}(x_{i},r).
\end{equation*}
The iteration of \eqref{eq:lengthbound_final} then gives
\[
    \mathcal{H}^1 (\partial^M \tilde{F}\cap \overline{\Omega})
    \le \mathcal{H}^1 (\partial^M F\cap \overline{\Omega}) + \sum_{k=0}^{N-1} 16 C \mathcal{H}^1\left(\overline{B}_{\sfd_{\Omega}}(x_{k+1},r)\setminus \bigcup_{i=1}^{k} \overline{B}_{\sfd_{\Omega}}(x_{i},r) \cap \partial \Omega \cap \partial^M F_k\right).
\]
Using \eqref{eq:decreasingboundary}, we finally have
\[
 \mathcal{H}^1 (\partial^M \tilde{F}\cap \overline{\Omega})
 \le \mathcal{H}^1 (\partial^M F\cap \overline{\Omega}) + 16 C\mathcal H^1(\partial^M F\cap \partial \Omega)
\]
and thus (4) is proven.
\end{proof}

For the next proposition we recall a Lusin-type lemma on uniform convergence for monotone measurable functions.
\begin{lemma}
\label{lemma:lusin_unif_conv}
Let $g_n \colon [0,1] \to \mathbb{R}$ be an increasing (or decreasing) sequence of measurable functions pointwise converging to $g \colon [0,1] \to \mathbb{R}$. For every $\varepsilon>0$, there exists a compact set $K \subset [0,1]$ such that $\mathcal L^1([0,1] \setminus K) \le \varepsilon$ for which $g_n \to g$ uniformly on $K$.
\end{lemma}

\begin{proposition}\label{prop:kicking_for_BV}
Let $E \subset \R^2$ be a compact ${BV}$-extension set. Let $\{ \Omega_i \}_{i\in I}$ be the connected components of $\R^2 \setminus E$. Then for $u \in W^{1,1}(E)$ there exists $\bar{u} \in {BV}(\R^2)$ such that
\begin{enumerate}
\item $\lvert D \bar{u} \rvert (\partial \Omega_i) = 0$ for all $i \in I$,
\item $\lvert D \bar{u} \rvert (\Omega_i \cap \cdot) \ll \mathcal{L}^2 \restr{\Omega_i}$ for all $i \in I$,
\item $\lVert \bar{u} \rVert_{{BV}(\R^2)} \leq c\lVert u \rVert_{{BV}(E)}$,
\item $\bar u = u$ on $E$. 
\end{enumerate}
\end{proposition}
\begin{proof}
 

Let us recall that also the set \(\overline {E^1}\) is a BV-extension set by Proposition \ref{prop_closedrepr}. We denote by \(\{\Omega^1_i\}_{i\in J}\) the family of connected components of \(\R^2\setminus \overline{E^1}\), where $J \subset \mathbb N$.  Then by Lemma \ref{lma:small_intersect}
\begin{equation}\label{eq:acc_finite}
\mathcal{H}^1(\partial \Omega_i^1 \cap \partial \Omega_j^1) = 0.
\end{equation}

Our aim is to exploit \eqref{eq:acc_finite} and prove the following claim.

{\color{blue}\textsc{Claim}}: Given $u \in W^{1,1}(\overline{E^1})$ there exists $\bar{u} \in {BV}(\R^2)$ such that (1)--(4) in the statement of the proposition are satisfied for $\bar{u}$ and the family $\{ \Omega_j^1 \}_{j \in J}$. 

Assuming that the claim is true, we conclude the proof as follows. To avoid confusion, we relabel the properties given by the claim as (1')--(4').
Given $u \in W^{1,1}(E)$, $\overline{E^1}\subset E$ because $E$ is closed. So the restriction of $u$ to $\overline{E^1}$ belongs to $W^{1,1}(\overline{E^1})$.
By the claim, there exists $\bar{u} \in BV(\mathbb{R}^2)$ which satisfies (1')--(4'). We prove that $\bar{u}$ and the sets $\{\Omega_i\}_{i \in I}$ satisfy the conclusion of the proposition. 

First of all, from $\mathcal{L}^2(\overline{E^1}\Delta E)=0$ and the property (4') we have $\bar{u}=u$ $\mathcal{L}^2$-a.e. on $E$, and so (4) holds.
We get (3) from the estimate
\[
\|\bar{u}\|_{BV(\mathbb{R}^2)} \le C \| u \|_{BV(\overline{E^1})} \le C \| u \|_{BV(E)},
\]
where the first inequality follows by (3') and the second one because the restriction from $BV(E)$ to $BV(\overline{E^1})$ is $1$-Lipschitz.

It remains to prove (1) and (2). We notice that for every \(i\in I\) there exists \(j(i)\in J\) such that \(\overline{\Omega_i}\subset \overline{\Omega^1_{j(i)}}\). Since $|D \bar{u}|(\Omega_i \cap \cdot) \le |D \bar{u}|(\Omega_{j(i)} \cap \cdot)$, this implies property (2). Property (2') gives that $|D \bar{u}|(\Omega_{j(i)}^1 \cap \partial \Omega_i)=0$, because $\mathcal{L}^2(\partial \Omega_i)=0$. This fact together with (1') proves (1).

\smallskip

We now prove the claim. Take $u \in W^{1,1}(\overline{E^1}) \subset BV(\overline{E^1})$.
We can assume without loss of generality that $0 \le u \le 1$.
 We extend $u$ to a function $\tilde{u} \in BV(\R^2)$ with $\|\tilde u\|_{BV(\R^2) }\le C\|u\|_{BV(\overline{E^1})}$.
Without loss of generality, we may assume that $\tilde u$ has compact support.


\textsc{{\color{blue}Step 1}} (Construction of an approximating sequence $\{u_n^i\}_n$ of the modification of $\tilde{u}$ in a fixed connected component $\Omega_i^1$):
Fix $i\in J$. Since $\tilde{u} \in BV(\mathbb{R}^2)$, by coarea formula we have
\[
 \int_{0}^{1} {\rm P}(\{\tilde u > t\})\,\d t= |D\tilde{u}|(\R^2) <\infty.
\]
Denote the superlevel set by $A_t = \{\tilde u > t\}$.
Thus, there exists $N \subset [0,1]$ such that $\mathcal{L}^1(N)=0$ and for every $t \in [0,1]\setminus N$ the set $A_t$ has finite perimeter. For every $t \in [0,1]\setminus N$, we can define $\widetilde{A}_t$ using Lemma $\ref{lem:kicking}$ with the choices $\Omega=\Omega_i^1$ and  $\varepsilon=2^{-i}\|\bar{u}\|_{L^1(\mathbb{R}^2)}$ so that 
\begin{enumerate}
\item[(i)] $\widetilde{A}_{t} \setminus \overline{\Omega^1_i} = A_{t} \setminus \overline{\Omega_i^1}$,
\item[(ii)] $\mathcal{H}^1 (\partial^M \widetilde{A}_{t}\cap \overline{\Omega^1_i}) \leq C\mathcal{H}^1(\partial^M A_{t}\cap \overline{\Omega^1_i})$,
\item[(iii)] $\mathcal{H}^1(\partial^M \widetilde{A}_{t} \cap \partial \Omega_i^1) = 0$,
\item[(iv)] $\mathcal{L}^2(\widetilde{A}_{t} \Delta A_{t})\leq  2^{-i}\|\tilde u\|_{L^1(\mathbb R^2)}$.
\end{enumerate}
For every $t \in N$ we set $\tilde{A}_t:= \emptyset$.

By item iii), we are in position to apply Lemma $\ref{lemma:lusin_unif_conv}$ for every $m \in \mathbb N$
with $g_k(t) := \Pe(\tilde{A}_t,B(\partial \Omega_i^1,2^{-k}))$, $g(t) := 0$, and $\varepsilon = 2^{-m}$. Let us denote by $K_m\subset [0,1]$ the set given by the lemma. Without loss of generality, we may assume that $K_m \subset K_{m+1}$ and $K_m \cap N =\emptyset$ for every $m$.
Let us also write, for every $m \in \mathbb N$,
$\tilde K_{m} = K_m \setminus K_{m-1}$.
We now define $I_{j,n} =  [(j-1)2^{-n},j2^{-n}]$ for every $j=1,\dots,2^{n}$ and $n \in \mathbb{N}$.


Next we define 
\[
S_{m,n}:=\{ j : |I_{j,n} \cap \tilde K_m| > 0 \} \subset \{1,\dots,2^{n}\}
\]
and for every $j \in S_{m,n}$ select $t_{j,m}^n \in I_{j,n} \cap \tilde K_{m}$ such that
\begin{equation}\label{eq:choice1}
 \Pe(A_{t_{j,m}^n}) \le \frac{{5}}{|I_{j,n} \cap \tilde K_{m}|}\int_{I_{j,n} \cap \tilde K_{m}} 
 \Pe(A_r)\,\d r
\end{equation}
and
\begin{equation}\label{eq:choice2}
 \Pe(A_{t_{j,m}^n}, \overline{\Omega^1_i}) \le \frac{{5}}{|I_{j,n} \cap \tilde K_{m}|}\int_{I_{j,n} \cap \tilde K_{m}} 
 \Pe(A_r, \overline{\Omega^1_i})\,\d r.
\end{equation}

Keeping in mind that all the above defined quantities depend on $i$, we define 
\[
u^i_n:= \sum_{m=1}^{\infty} \sum_{j \in S_{m,n}}|I_{j,n} \cap \tilde K_{m}|\chi_{\tilde{A}_{t_{j,m}^n}}.
\]
By (i) and (ii)  we have that 
\begin{equation}
    \label{eq:boundedness_extension_superlevel}
    {\rm P}({\tilde{A}_{t_{j,m}^n}}) \le C \Pe({{A}_{t_{j,m}^n}}).
\end{equation}
Thus,
\begin{equation}\label{eq:bound1}
\begin{split}
    |Du^i_n|(\R^2) & \le  \sum_{m=1}^{\infty} \sum_{j \in S_{m,n}}|I_{j,n} \cap \tilde K_{m}|\Pe({\tilde{A}_{t_{j,m}^n}})\\ 
    & \le \sum_{m=1}^{\infty} \sum_{j \in S_{m,n}}C|I_{j,n} \cap \tilde K_{m}|\Pe({{A}_{t_{j,m}^n}})\\
    &\stackrel{\eqref{eq:choice1}}{\le} \sum_{m=1}^{\infty} \sum_{j \in S_{m,n}}3 C \int_{I_{j,n} \cap \tilde K_{m}}\Pe({{A}_{r}})\,\d r
 = 3C |D\tilde u|(\R^2),
  \end{split}
\end{equation}
whereas by (ii) we get
\begin{equation}\label{eq:bound2}
\begin{split}
|D u^i_n|(\overline{\Omega^1_i}) = &
\sum_{m=1}^{\infty} \sum_{j \in S_{m,n}}|I_{j,n} \cap \tilde K_{m}|\Pe(\tilde{A}_{t_{j,m}^n},\overline{\Omega^1_i}) \\
 &\le \sum_{m=1}^{\infty} \sum_{j \in S_{m,n}}C |I_{j,n} \cap \tilde K_{m}|\Pe({A}_{t_{j,m}^n},\overline{\Omega^1_i}) \\
 & \stackrel{\eqref{eq:choice2}}{\le}  \sum_{m=1}^{\infty} \sum_{j \in S_{m,n}}3C \int_{I_{j,n} \cap \tilde K_{m}}\Pe({{A}_{r}},\overline{\Omega}_i)\,\d r
 = 3C |D\tilde u|(\overline{\Omega_i^1}).
 \end{split}
\end{equation}

\textsc{{\color{blue}Step 2}} (Limit of the approximating sequence and construction of the modification $u^i$ of $\tilde{u}$ on $\Omega_i^1$):

We claim that $\| u^i_n - u \|_{L^\infty(E)} \le 2^{-n}$.
Indeed, given $x \in \overline{E^1}$, there exists $j_0=j_0(x)$ such that $(j_0-1)2^{-n} \le u(x) \le j_0 2^{-n}$.
We observe that, given $m \in \mathbb{N}$ and $j \in S_{m,n}$, if $t_{j,m}^n \le (j_0-1)2^{-n}$, then $\chi_{\tilde{A}_{t_{j,m}^n}}(x)= \chi_{A_{t_{j,m}^n}}(x)=1$. This implies that
\begin{equation}
    u^i_n(x)\ge \sum_{m=1}^{\infty} \sum_{j \in S_{m,n}:\,t_{j,m}^n\le (j_0-1)2^{-n}}|I_{j,n} \cap \tilde K_{m}| = (j_0-1)2^{-n}.
\end{equation}
Similarly, if $t_{j,m}^n \ge j_0 2^{-n}$, then $\chi_{\tilde{A}_{t_{j,m}^n}}(x)=0$. Thus,
\begin{equation}
    u^i_n(x)\le  \sum_{m=1}^{\infty} \sum_{j \in S_{m,n}:\, t_{j,m}^n < j_0\,2^{-n}}|I_{j,n} \cap \tilde K_{m}|=j_0\, 2^{-n}.
\end{equation}
This proves the claim.

By uniform convergence of $g_k$ to $0$ on $K_m$, we know that for every $m \in \mathbb N$ there exists $\bar{k} = \bar{k}(m) \in \mathbb N$ such that $g_k(t) \le 2^{-m}$ for all $k \ge \bar{k}$ and all $t \in K_m$.
We define for all $\delta>0$
\[ w(\delta):= \sup \left\{ \int_D  
\Pe(A_s)\, \d s:\, D \subset [0,1]\text{ Borel with }\,\mathcal{L}^1(D) \le \delta \right\}.\]
We notice that $w(\delta)<\infty$ for every $\delta >0$, since by coarea formula $w(\delta)\le|D \tilde{u}|(\mathbb{R}^2)<\infty$.
We compute, for any $m_0,n \in \mathbb N$ and for $k \ge \bar{k}(m_0)$
\begin{equation}
\label{eq:total_variation_tub_neigh}
\begin{aligned}
|D u^i_n|& (B(\partial \Omega_i, 2^{-k})) \le \sum_{m=1}^{\infty} \sum_{j \in S_{m,n}}|I_{j,n} \cap \tilde K_{m}|\Pe(\tilde{A}_{t_{j,m}^n},B(\partial \Omega_i, 2^{-k}))\\
& \le \sum_{m=1}^{m_0} \sum_{j \in S_{m,n}}|I_{j,n} \cap \tilde K_{m}|\Pe(\tilde{A}_{t_{j,m}^n},B(\partial \Omega_i, 2^{-k}))
+
\sum_{m=m_0+1}^{\infty} \sum_{j \in S_{m,n}}|I_{j,n} \cap \tilde K_{m}|\Pe(\tilde{A}_{t_{j,m}^n})\\
& \stackrel{\eqref{eq:boundedness_extension_superlevel}}{\le} \sum_{m=1}^{m_0} \sum_{j \in S_{m,n}}|I_{j,n} \cap \tilde K_{m}|2^{-m_0}
+ \sum_{m=m_0+1}^{\infty} \sum_{j \in S_{m,n}}|I_{j,n} \cap \tilde K_{m}|C\Pe(A_{t_{j,m}^n})\\
& \stackrel{\eqref{eq:choice1}}{\le} \sum_{m=1}^{m_0} \sum_{j \in S_{m,n}}|I_{j,n} \cap \tilde K_{m}|2^{-m_0}
+
\sum_{m=m_0+1}^{\infty} \sum_{j \in S_{m,n}}C\int_{I_{j,n} \cap \tilde K_{m}} 
\Pe(A_r)\,\d r\\
& \le 2^{-m_0} + Cw(2^{-m_0}).
\end{aligned}
\end{equation}
In the third inequality, we used that $t_{j,m}^n \in I_{j,n} \cap \tilde{K}_m \subset K_{m_0}$ and in the last inequality we used that $K_{m_0} \subset [0,1]$ since $0 \le u \le 1$.

Since $0\le u^i_n \le 1$ for every $n$, 
we have
\[
\sup_n \| u^i_n \|_{L^2(\mu)}< \infty.
\]
where $\mu$ is a nonnegative finite Borel measure such that $\mathcal{L}^2 \ll \mu \ll \mathcal{L}^2$.
Hence, there exists a subsequence $u^i_{n_k} \rightharpoonup u^i$ in $L^2(\mu)$. We apply Mazur's lemma to this subsequence to obtain a sequence $v^i_m = \sum_{j=m}^{N_m} \lambda_{m,j}u^i_{n_{j}}$ with $\sum_{j=m}^{N_m} \lambda_{m,j} =1$ such that $v^i_m \to u^i$ in $L^2(\mu)$ as $m \to \infty$. Thus, $v_m^i \to u^i$ in $L^1_{\rm loc}(\mathbb{R}^2)$.


\textsc{{\color{blue}Step 3}} (BV estimates on $u^i$): By item (iv), we have that  
\begin{equation}
\label{eq:bounding_uni}
    \begin{aligned}
    \|u^i_n\|_{L^1(\mathbb{R}^2)}&\le  \sum_{m=1}^{\infty} \sum_{j \in S_{m,n}}|I_{j,n} \cap \tilde K_{m}| \mathcal{L}^2(\tilde{A}_{t_{j,m}^n})\\
    &\le \sum_{m=1}^{\infty} \sum_{j \in S_{m,n}}|I_{j,n} \cap \tilde K_{m}| \left(\mathcal{L}^2(A_{(j-1)2^{-n}}) + 2^{-i} \|\tilde{u}\|_{L^1(\mathbb{R}^2)}\right)\\
    & = \sum_{j=1}^{2^{n}} 2^{-n}\left(\mathcal{L}^2(A_{(j-1)2^{-n}}) + 2^{-i} \|\tilde{u}\|_{L^1(\mathbb{R}^2)}\right)\\
    & \le \sum_{j=1}^{2^{n}} 2^{-n}\left(\mathcal{L}^2(A_{j2^{-n}}) + 2^{-i} \|\tilde{u}\|_{L^1(\mathbb{R}^2)}\right) + 2^{-n}\mathcal{L}^2(A_{0})\\
    & \le  \int_0^1 \mathcal{L}^2(A_r)\,\d r + 2^{-i} \|\tilde{u}\|_{L^1(\mathbb{R}^2)} + 2^{-n}\mathcal{L}^2(A_{0})\\
    & = (1 + 2^{-i}) \|\tilde{u}\|_{L^1(\mathbb{R}^2)} + 2^{-n}\mathcal{L}^2(A_{0}),
    \end{aligned}
\end{equation}
where the last inequality follows from Cavalieri's formula. Since $\tilde u$ has compact support, $\mathcal{L}^2(A_{0}) < \infty$ and thus for large enough $n$, 
\eqref{eq:bounding_uni} gives
\begin{equation}
\label{eq:bounding_uni2}
\|u^i_n\|_{L^1(\mathbb{R}^2)} \le(1+ 2^{-i+1}) \|\tilde{u}\|_{L^1(\mathbb{R}^2)}
\end{equation}
By \eqref{eq:bound1} and \eqref{eq:bounding_uni2}, for large enough $n$ we have the crude estimate
\[
\|u^i_n\|_{BV(\R^2)} \leq C \|\tilde u\|_{BV(\R^2)}
\]
and thus also for $v_n^i$.
This together with the lower semicontinuity of total variation leads to
\[
\|u^i\|_{BV(\R^2)} \leq C \|\tilde u\|_{BV(\R^2)}. 
\]

For any $n$, $m$ it holds from \eqref{eq:total_variation_tub_neigh} that for $k \ge \bar{k}(m)$ 
\[
|D v^i_n|(B(\partial \Omega_i^1, 2^{-k}))\le 2^{-m} +  C w(2^{-m}).
\]
Hence, we have
\[ |D u^i|(\partial \Omega_i^1) \le |D u^i|(B(\partial \Omega_i^1, 2^{-k})) \le 2^{-m} + C w(2^{-m}),\]
where in the last inequality we used the lower semicontinuity of total variation on open sets with respect to $L^1_{\rm loc}$ convergence applied to the open set $B(\partial \Omega_i^1, 2^{-k})$. By taking the limit as $m \to +\infty$, we get 
\begin{equation}\label{eq:boundaryuizero}
|D u^i|(\partial \Omega_i^1) = 0.
\end{equation}
From (iv), it follows that  $\limi_{n \to \infty} \|\tilde{u}-u_n^i\|_{L^1(\mathbb{R}^2)} \le 2^{-i} \| \tilde{u} \|_{L^1(\mathbb{R}^2)}$. Up to passing to $v_n^i$ with a convex combination and lower semicontinuity with respect with $L^1_{\rm loc}$ convergence, we get

\begin{equation}\label{eq:L1diff}
\|u^i - \tilde u\|_{L^1(\mathbb{R}^2)} \le 2^{-i}\|\tilde u\|_{L^1(\mathbb{R}^2)}.
\end{equation}
Moreover, by \eqref{eq:bound2} and the lower semicontinuity of total variation on the open set $\Omega_i^1$, together with \eqref{eq:boundaryuizero}, we have
\begin{equation}\label{eq:BVincrease}
 | D u^i|(\overline{\Omega^1_i}) \le C |D \tilde u|(\overline{\Omega^1_i}).
\end{equation}




\textsc{{\color{blue}Step 4}} (Construction of the function $\bar{u}$):
We can repeat the construction for every $i \in J$. Now, we define 
\[
\bar{u}(x) := 
\begin{cases}
u(x),& \text{if } x\in \overline{E_1} \\
{u}^i(x),& \text{if } x\in \Omega_i^1\text{ for some }i \in J.
\end{cases}
\]
Let us check that
\begin{enumerate}
\item[(a)] $\lvert D\bar{u} \rvert (\partial \Omega_i^1) = 0$ for each $i\in J$;
\item[(b)] $\lVert \bar{u} \rVert_{BV(\R^2)} \leq C \lVert u \rVert_{BV(\overline{E^1})}$.
\end{enumerate}
Towards this, we define ${v}^j := \left({u}^j- \tilde{u} \right) \chi_{\Omega_j^1}$ and
\begin{equation}\label{eq:barudecomposition}
\bar{u} = \tilde{u} + \sum_{j\in J} v^j =\tilde{u} \chi_{\mathbb{R}^2\setminus \Omega_i^1} + u^i \chi_{\Omega_i^1} + \sum_{j \neq i} v^j = u^i + \sum_{j \neq i} v^j,
\end{equation}
where in the last equality we used that, because of item (i), $\tilde{u} \chi_{\mathbb{R}^2\setminus \Omega_i^1}= u_n^i \chi_{\mathbb{R}^2\setminus \Omega_i^1}$.

Now for item (b), notice that $v^j = \left({u}^j - \tilde{u} \right)\chi_{\Omega_j^1} = {u}^j - \tilde{u}$, and thus
from \eqref{eq:BVincrease} we get
\[
|D v^j|(\R^2) \leq  |D{u}^j|(\overline{\Omega^1_j}) + |D\tilde{u}|(\overline{\Omega^1_j}) \leq C|D\tilde{u}|(\overline{\Omega^1_j})
\]
and from \eqref{eq:L1diff}
\[
 \|v^j\|_{L^1(\R^2)} =  \|u^j - \tilde u\|_{L^1(\mathbb{R}^2)} \le 2^{-j}\|\tilde u\|_{L^1(\mathbb{R}^2)}.
\]
Combining the above, we have 
\[
\lVert \bar{u} \rVert_{BV(\R^2)} \leq \lVert \tilde{u} \rVert_{BV(\R^2)} + C \sum_{j\in J} \lVert \tilde{u} \rVert_{BV(\overline{\Omega^1_j})} \leq (C+1)\lVert \tilde{u} \rVert_{BV(\R^2)} \leq C\lVert u \rVert_{BV(\overline{E^1})},
\]
where in the second inequality we used \eqref{eq:acc_finite}. 

For item (a), 
again by \eqref{eq:acc_finite} and the coarea formula, we have
\[
 \lvert D{v}^j \rvert (\partial \Omega_i^1) = 0
\]
for $j \ne i$, and by \eqref{eq:boundaryuizero}
\[
\lvert Du^i \rvert (\partial \Omega_i^1) = 0.
\]
Thus, by the decomposition after the last equality in \eqref{eq:barudecomposition},
we have (a).
In order to obtain the property (2) for the components \(\{\Omega_i^1\}_{i\in J}\), we use the smoothing operator of Proposition \ref{prop:smoothing} for the open set $\Omega = \R^2 \setminus \overline{E^1}$.
\end{proof}

We can now conclude the proof of Theorem \ref{thm:main}. Notice that although Proposition \ref{prop:kicking_for_BV} holds also for infinitely connected closed sets, it is still open in if the conclusion of Theorem \ref{thm:main} holds in this general setting. As will be seen in Example \ref{ex:method_fails} after the proof, the approach via Proposition \ref{prop:kicking_for_BV} cannot directly give the claim in the general setting.

\begin{proof}[Proof of Theorem \ref{thm:main}]
Let $E \subset\mathbb R^2$ be a compact $BV$-extension set with $\mathbb R^2\setminus E$ having finitely many connected components, say $m \in \mathbb{N}$. Our aim is to show that $E$ is a $W^{1,1}$-extension set. To this end, let $u \in W^{1,1}(E)$. Since $\mathbb R^2 \setminus E$ has finitely many components $\Omega_i$, we have
\[
\partial E = \bigcup_{i=1}^m\partial \Omega_i.
\]
Thus by Proposition \ref{prop:kicking_for_BV}, there exists $\tilde u \in BV(\mathbb R^2)$ with $|D\tilde u|(\partial E) = 0$, $\tilde u\restr{\mathbb R^2\setminus E} \in W^{1,1}(\mathbb R^2 \setminus E)$, and $\|\tilde u\|_{BV(\mathbb R^2)} \le C \|u\|_{BV(E)}$, where $C$ is independent of $u$. Thus, by Lemma \ref{lemma:fromlocal_to_global_W11}, we have $\tilde u \in W^{1,1}(\mathbb R^2)$ with
\[
\|\tilde u\|_{W^{1,1}(\mathbb R^2)} =\|\tilde u\|_{BV(\mathbb R^2)} \le C \|u\|_{BV(E)}
\]
proving the claim.
\end{proof}

Let us end with an example showing that the conclusion of Proposition \ref{prop:kicking_for_BV} is by itself not enough 
to show that $\tilde u \in W^{1,1}(\mathbb R^2)$.

\begin{example}\label{ex:method_fails}
 Let us give an example of a compact set $E \subset \mathbb R^2$ and a function $u\in BV(\mathbb R^2)\setminus W^{1,1}(\mathbb R^2)$ so that 
 $u\restr{E} \in W^{1,1}(E)$, and $|Du|(\partial \Omega_i) = 0$  and $u\restr{\Omega_i}\in W^{1,1}(\Omega_i)$ for all connected components $\Omega_i$ of $\mathbb R^2\setminus E$.

 Let $C \subset [0,1]$ be a fat Cantor set and denote by $\{U_i\}$ the bounded open connected components of $\mathbb R \setminus C$. Let $f \in BV(\mathbb R)\setminus W^{1,1}(\mathbb R)$ be a continuous increasing function  with $f(x) = 0$ for all $x \le 0$, $f(x) = 1$ for all $x \ge 1$, and $f$ constant on each $U_i$. (For example, one can take a Cantor staircase function constructed with a further Cantor subset of $C$ of zero Lebesgue measure.)

 Let us then define $E = [0,1]^2\cup (C \times [1,4]) \cup ([0,1]\times [4,5])$ and set 
 \[
 u(x,y) = \max(0,1-\textrm{dist}(x,[0,1]))\max(0,1-\textrm{dist}(y,[2,3]))f(x).
 \]

 Let us then check the claimed properties for $E$ and $u$. Since $v \in BV(\mathbb R^2)$, also $u \in BV(\mathbb R^2)$. On every horizontal line-segment intersecting the set $[0,1]\times(2,3)$ the function $u$ is continuous, but not absolutely continuous. Thus $u \notin W^{1,1}(\mathbb R^2)$. Since $u$ is a Lipschitz function on each vertical segment in $C\times [1,4]$ and zero in the rest of $E$, we have $u\restr{E} \in W^{1,1}(E)$. Since $u$ is Lipschitz in each connected component $\Omega_i$ of $\mathbb R^2\setminus E$, we have $u\restr{\Omega_i}\in W^{1,1}(\Omega_i)$ for all $\Omega_i$. Moreover, the boundary $\partial \Omega_i$ is rectifiable for all $\Omega_i$ and the function $u$ has only the absolutely continuous and Cantor parts, $|Du|(\partial \Omega_i) = 0$ for all $\Omega_i$.

 The point of taking a fat Cantor set $C$ in the construction above is to also have $\overline{E^1} = E$. In other words, $E$ does not have a representative with nicer properties.
\end{example}

%

\bibliographystyle{amsplain}
\bibliography{biblio.bib}

\end{document}